\documentclass[11pt]{article}

\usepackage[utf8]{inputenc}
\usepackage[english]{babel}
\usepackage{amsmath,amssymb,amsfonts,amsthm}
\usepackage{multirow}
\usepackage{xcolor}
\usepackage{hyperref}
\hypersetup{colorlinks=true,linkcolor=blue,filecolor=blue,citecolor = blue,urlcolor=blue}

\usepackage{array}

\usepackage{stmaryrd}

\usepackage[perpage]{footmisc}

\usepackage{footnote}
\makesavenoteenv{tabular}
\makesavenoteenv{table}

\usepackage{lscape}
\usepackage{pdflscape}

\usepackage{graphicx}
\usepackage[scale=0.75]{geometry}

\usepackage{enumitem}

\usepackage[font=footnotesize,labelfont=bf]{caption}

\usepackage{subfig} 

\theoremstyle{plain}
\newtheorem{theorem}{Theorem}
\newtheorem{lemma}{Lemma}
\newtheorem{proposition}{Proposition}
\newtheorem{corollary}{Corollary}

\theoremstyle{definition}
\newtheorem{definition}{Definition}
\newtheorem*{example}{Example}

\theoremstyle{remark}
\newtheorem*{remark}{Remark}

\newcommand{\prob}[1]{\mathbb{P}\left(#1\right)}
\newcommand{\esp}[1]{\mathbf{E}\left[#1\right]}
\newcommand{\espp}[2]{\mathbf{E}_{#1}\left[#2\right]}

\newcommand{\eqd}{\stackrel{(d)}{=}}
\newcommand{\ind}[1]{{\bf 1}_{#1}}
\newcommand{\borel}[1]{\mathcal{B}\left( #1 \right)}
\newcommand{\di}{\mathrm{d}}
\newcommand{\leb}{\mathrm{Leb}}
\renewcommand{\tilde}[1]{\widetilde{#1}}

\renewcommand{\emptyset}{\varnothing}

\newcommand{\dirac}[1]{\delta_{#1}}
\newcommand{\unif}[1]{\mathrm{Unif}\left(#1\right)}
\newcommand{\ber}[1]{\mathrm{Ber}\left(#1\right)}
\newcommand{\bin}[1]{\mathrm{Bin}\left(#1\right)}
\newcommand{\poi}[1]{\mathrm{Poi}\left(#1\right)}
\newcommand{\geom}[1]{\mathrm{Geom}\left(#1\right)}
\newcommand{\expo}[1]{\mathrm{Exp}\left(#1\right)}
\newcommand{\gam}[1]{\mathrm{Gamma}\left(#1\right)}
\newcommand{\bet}[1]{\mathrm{Beta}\left(#1\right)}
\newcommand{\norm}[1]{\mathcal{N}\left(#1\right)}

\newcommand{\NN}{\mathbb{N}}
\newcommand{\ZZ}{\mathbb{Z}}
\newcommand{\RR}{\mathbb{R}}
\renewcommand{\l}{\ell}

\newcommand{\rot}[1]{\widehat{#1}}
\newcommand{\rotop}{\wedge}
\newcommand{\dessin}{\mathcal{D}}
\newcommand{\vad}{{\mathbf{D}}} 

\newcommand{\GVH}{h} 
\newcommand{\MGVH}{\eta} 
\newcommand{\RN}{Radon--Nikodym}

\newcommand{\diff}[1]{#1} 

\newcounter{exampleTab}
\newcommand{\numex}{\refstepcounter{exampleTab}\theexampleTab}

\let\originalleft\left
\let\originalright\right
\renewcommand{\left}{\mathopen{}\mathclose\bgroup\originalleft}
\renewcommand{\right}{\aftergroup\egroup\originalright}

\newcommand{\rb}{-1mm}
\newcommand{\VE}{\raisebox{\rb}{\includegraphics{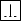}}} 
\newcommand{\VS}{\raisebox{\rb}{\includegraphics{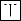}}} 
\newcommand{\VB}{\raisebox{\rb}{\includegraphics{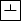}}} 
\newcommand{\VA}{\raisebox{\rb}{\includegraphics{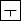}}} 
\newcommand{\VT}{\raisebox{\rb}{\includegraphics{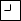}}} 
\newcommand{\HE}{\raisebox{\rb}{\includegraphics{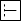}}} 
\newcommand{\HS}{\raisebox{\rb}{\includegraphics{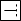}}} 
\newcommand{\HB}{\raisebox{\rb}{\includegraphics{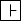}}} 
\newcommand{\HA}{\raisebox{\rb}{\includegraphics{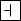}}} 
\newcommand{\HT}{\raisebox{\rb}{\includegraphics{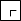}}} 
\newcommand{\OB}{\raisebox{\rb}{\includegraphics{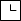}}} 
\newcommand{\OA}{\raisebox{\rb}{\includegraphics{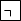}}} 
\newcommand{\CC}{\raisebox{\rb}{\includegraphics{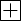}}} 

\newcommand{\rbt}{-0.7mm}
\newcommand{\sct}{0.7}
\newcommand{\tVE}{\raisebox{\rbt}{\includegraphics[scale=\sct]{Images/Icone/VE.pdf}}} 
\newcommand{\tVS}{\raisebox{\rbt}{\includegraphics[scale=\sct]{Images/Icone/VS.pdf}}} 
\newcommand{\tVB}{\raisebox{\rbt}{\includegraphics[scale=\sct]{Images/Icone/VB.pdf}}} 
\newcommand{\tVA}{\raisebox{\rbt}{\includegraphics[scale=\sct]{Images/Icone/VA.pdf}}} 
\newcommand{\tVT}{\raisebox{\rbt}{\includegraphics[scale=\sct]{Images/Icone/VT.pdf}}} 
\newcommand{\tHB}{\raisebox{\rbt}{\includegraphics[scale=\sct]{Images/Icone/HB.pdf}}} 
\newcommand{\tHA}{\raisebox{\rbt}{\includegraphics[scale=\sct]{Images/Icone/HA.pdf}}} 
\newcommand{\tHT}{\raisebox{\rbt}{\includegraphics[scale=\sct]{Images/Icone/HT.pdf}}} 

\newcommand{\tOB}{\raisebox{\rbt}{\includegraphics[scale=\sct]{Images/Icone/OB.pdf}}} 
\newcommand{\tOA}{\raisebox{\rbt}{\includegraphics[scale=\sct]{Images/Icone/OA.pdf}}} 
\newcommand{\tCC}{\raisebox{\rbt}{\includegraphics[scale=\sct]{Images/Icone/CC.pdf}}} 

\newcommand{\rbtt}{-0.65mm}
\newcommand{\sctt}{0.6}
\newcommand{\ttVE}{\raisebox{\rbtt}{\includegraphics[scale=\sctt]{Images/Icone/VE.pdf}}} 
\newcommand{\ttVB}{\raisebox{\rbtt}{\includegraphics[scale=\sctt]{Images/Icone/VB.pdf}}} 
\newcommand{\ttVA}{\raisebox{\rbtt}{\includegraphics[scale=\sctt]{Images/Icone/VA.pdf}}} 
\newcommand{\ttVT}{\raisebox{\rbtt}{\includegraphics[scale=\sctt]{Images/Icone/VT.pdf}}} 
\newcommand{\ttHE}{\raisebox{\rbtt}{\includegraphics[scale=\sctt]{Images/Icone/HE.pdf}}} 
\newcommand{\ttHB}{\raisebox{\rbtt}{\includegraphics[scale=\sctt]{Images/Icone/HB.pdf}}} 
\newcommand{\ttHA}{\raisebox{\rbtt}{\includegraphics[scale=\sctt]{Images/Icone/HA.pdf}}} 
\newcommand{\ttHT}{\raisebox{\rbtt}{\includegraphics[scale=\sctt]{Images/Icone/HT.pdf}}} 
\newcommand{\ttOB}{\raisebox{\rbtt}{\includegraphics[scale=\sctt]{Images/Icone/OB.pdf}}} 
\newcommand{\ttCC}{\raisebox{\rbtt}{\includegraphics[scale=\sctt]{Images/Icone/CC.pdf}}} 

\newcommand{\paral}{/\hspace{-0.7mm}/}

\title{Reversible Poisson-Kirchhoff Systems}

\author{Alexandre Boyer\thanks{LMO, Université Paris-Saclay. Email: \texttt{alexandre.boyer.math@gmail.com}}\ , Jérôme Casse\thanks{LMO, Université Paris-Saclay. Email: \texttt{jerome.casse.math@gmail.com}}\ , Nathanaël Enriquez\thanks{LMO, Université Paris-Saclay. Email: \texttt{nathanael.enriquez@universite-paris-saclay.fr}}\ \ and Arvind Singh\thanks{CNRS. LMO, Université Paris-Saclay. Email: \texttt{arvind.singh@universite-paris-saclay.fr}}}
\date{\today}

\begin{document}

\maketitle

\begin{abstract}
We define a general class of random systems of horizontal and vertical weighted broken lines on the quarter plane whose distribution are proved to be translation invariant. This invariance stems from a reversibility property of the model. This class of systems generalizes several classical processes of the same kind, such as Hammersley's broken line processes involved in Last Passage Percolation theory or such as the six-vertex model for some special sets of parameters. The novelty comes here from the introduction of a weight associated with each line. The lines are initially generated by spatially homogeneous weighted Poisson Point Process and their evolution (turn, split, crossing) are ruled by a Markovian dynamics which preserves Kirchhoff's node law for the line weights at each intersection. Among others, we derive some new explicit invariant measures for some bullet models as well as new reversible properties for some six-vertex models with an external electromagnetic field.
\end{abstract}

\noindent \textbf{AMS Classification 2020:} 82C23, 60G10, 60G55.

\begin{figure*}[h]
    \begin{center}
    \includegraphics[height=6cm]{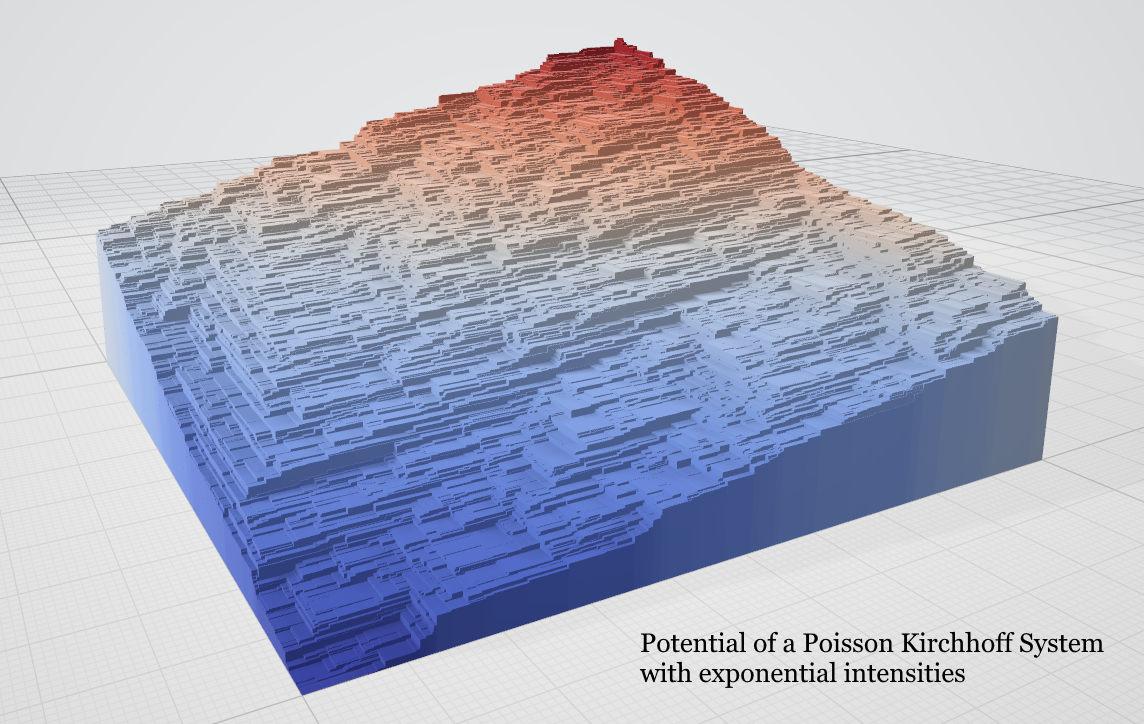}
    \end{center}
\end{figure*}

\newpage

\section{Introduction}\label{sec:intro}
In his seminal work \cite{Hammersley72}, Hammersley introduced its now famous \emph{broken line process} as a mean to study the length of the longest increasing sequence in a random permutation. This model of Last Passage Percolation (LPP) enjoys many remarkable properties and has since been thoroughly scrutinized~\cite{Rost81,Seppalainen09}. One possible construction of Hammersley's process on the quarter plane $[0,\infty)^2$ goes as follow: consider a unit intensity Poisson Point Process (PPP) on $[0,\infty)^2$. Each atom of the point process ``emits'' a pair of particle/anti-particle with the particle of charge $+1$ moving horizontally to the right and the antiparticle with charge $-1$ moving upward. When the traces of two particles of opposite charge meet, they both disappear. Then, the collection of all traces obtained with this procedure is exactly the Hammersley's broken line process on the quarter plane (see Figure~\ref{fig:CG-fleche} for an illustration of the construction). Let us note that, in view of this construction, the system may be called ``conservative'' in the sense that the total charge of the system remains null since particles and antiparticles appear and disappear simultaneously.

\begin{figure}
    \begin{center}
    \includegraphics[width=6cm]{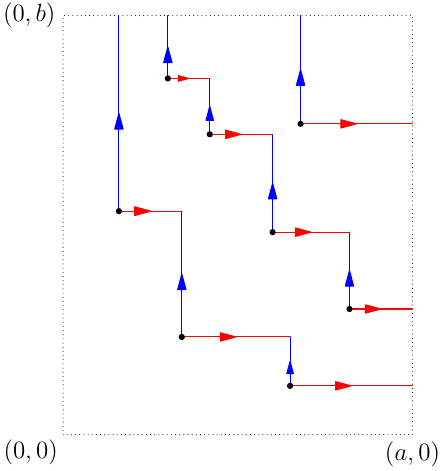}    
    \caption{A realisation of Hammersley's broken line process in the rectangle $[0,a]\times[0,b]$. The traces of particles with charge $+1$ are represented in red and those of the antiparticles with charge $-1$ are represented in blue.}
    \label{fig:CG-fleche}
    \end{center}
\end{figure}

In this paper, we introduce a new class of random processes which we call \emph{Poisson-Kirchhoff Systems} (PKS) that generalize the construction described above. Those processes consist again of random collections of weighted horizontal and vertical broken lines living on the quarter plane $[0,\infty)^2$. As for the Hammersley's broken line process, one may think of these lines as being the traces of ``charged'' particles moving either horizontally (i.e.\ increasing their $x$-coordinate) or vertically (i.e.\ increasing their $y$-coordinate). However, in this new class of processes, particles may hold arbitrary charges and may randomly turn, split or coalesce  according to a special Markovian dynamics which is still conservative in the sense that the total charge remains constant. We show in this paper that, when the parameters of the dynamics take a particular form, the PKS process is spatially reversible. Then, it is possible to construct a translation invariant PKS process on the whole plane whose marginal distribution along vertical and horizontal lines are (weighted) PPPs. 

\medskip

The paper is organized as follows. In Section~\ref{sec:PKS} we define the PKS process in a general setting and prove its existence under a uniform boundedness assumption on the parameters. 

In Section~\ref{sec:reversPKS}, we introduce a notion of reversibility for PKS processes which essentially says that the distribution of a PKS restricted to any rectangular box is invariant by a rotation of $180$ degrees. Then we present, in our main results, suitable conditions that guarantee the reversibility and therefore the invariance of PKS processes. We do it in three different frameworks according to whether the distribution of the line weights is absolutely continuous with respect to Lebesgue measure, discrete or arbitrary.

The proof of this reversibility property is carried out in Section~\ref{sec:proof}. The state space of PKS processes is quite complicated, and in order to deal with it, we introduce a family of parametrizations. It turns out that two different parametrizations of this family define the same volume form. We apply this result to two specific parametrizations: a first one associated to the dynamics of the PKS and the second one associated to its reverse dynamics. Once we have done it, a careful analysis shows that the densities associated to the dynamics and to the reversed one in their respective parametrizations coincide under the above-mentioned conditions. Interestingly, one can exploit this invariance result in order to extend the proof of the existence of the PKS to unbounded parameters.

In Section~\ref{sec:examples}, we first show how Kirchhoff's node law makes it possible to define a notion of potential function associated with the faces of the tessellation defined by a PKS. This potential function corresponds to the last passage times in LPP. We then collect several LPP models which can be mapped to PKS processes. In the sequel, we provide a (non-exhaustive) list of PKS processes obtained for specific distributions of the line weights. From this list, we recover several other classical models of statistical physics like bullet models~\cite{KRL95,BM20,HST21} or six-vertex models~\cite{Pauling35,Baxter72}. In particular, we exhibit some new explicit invariant measures for some bullet models as well as new reversible properties for some six-vertex models with an external electromagnetic field. Furthermore, the special cases of Gaussian or Poisson distributions for the line weights provide new models with explicit dynamics which might be worthy of further study.

Finally, in Section~\ref{sec:tess}, we look at basic geometric properties of the random tessellation of the quarter plane induced by a PKS, such as the mean number of connected components inside a rectangle, and the mean number of nodes of a typical connected component.

\section{Poisson-Kirchhoff systems}\label{sec:PKS}
The definition of a generic Poisson-Kirchhoff process relies on 9 parameters. First, let $\lambda_0$  be a non-negative number which will be referred to as the \emph{spontaneous creation rate}. Let $\lambda_V$ and $\lambda_H$ be two functions from $\RR$ to $\RR_+$, called \emph{vertical and horizontal split rate functions}. Let $\tau_V$ and $\tau_H$ be two functions from $\RR$ to $\RR_+$, called \emph{vertical and horizontal turn rate functions}. Let $p_0 \in [0,1]$ called the \emph{annihilation probability}, let also $p_V$ and $p_H$ be two functions from $\RR$ to $[0,1]$, called respectively \emph{vertical} and \emph{horizontal coalescence probability functions} that satisfy, for any $s \in \RR$,
\begin{equation}\label{eq:pprob}
p_V(s) + p_H(s) + p_0 \ind{s=0} \leq 1. 
\end{equation}
Finally, let $F = (F(s,\cdot) : s \in \RR)$ be a probability transition kernel on $\RR$, called the \emph{division kernel}, which satisfies:
\begin{itemize}
    \item The map $s \mapsto F(s,B)$ is $\borel{\RR}$-measurable for any Borel set $B \in \borel{\RR}$.
    \item $B \mapsto F(s,B)$ is a probability measure on $(\RR,\borel{\RR})$ for any $s \in \RR$.
\end{itemize}

\begin{figure}
    \begin{center}
    \includegraphics[width=3.5cm]{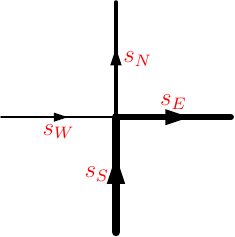}
    \caption{Kirchhoff's node law at a crossing: two lines are coming from the south and the west directions with respective weights $s_S$ and $s_W$. Lines exiting the intersection to the north and west direction have respective weights $s_N$ and $s_E$. The sum of weights entering and exiting the intersection is conserved: $s_S+s_W =s_N+s_E$.}
    \label{fig:Kirchhoff}
    \end{center}
\end{figure}

The collection $(\lambda_0,\lambda_V,\lambda_H,p_0,p_V,p_H,\tau_V,\tau_H,F)$ represents the parameters of the model. The three parameters $(\lambda_0,\lambda_V,\lambda_H)$ can be seen as splitting rates whereas $(p_0,p_V,p_H)$ can be seen as merging probabilities. We will see that these two sets of parameters play a dual role. The two parameters $(\tau_V,\tau_H)$ have a symmetric role and describe how often vertical and horizontal lines turn. Finally, the kernel $F$ describes the distribution of the weights when a line splits or when two lines meet and split again.\par \smallskip

We now define a random system of horizontal and vertical algebraic weighted lines inside the quarter plane $[0,\infty)^2$ which preserves Kirchhoff's node law at every intersection (w.r.t.\ their weights), as prescribed in Figure~\ref{fig:Kirchhoff}. As in the description of Hammersley's process in Section~\ref{sec:intro}, one can think of those lines as the traces of charged particles moving either to the right or upwards. Let us emphasize that, in our setting, the weight (i.e.\ charge) of a line may be positive, negative or even null.  \par

We define the \emph{initial condition} of our process by specifying the positions and weights of the vertical (resp.\ horizontal) lines that start from the $x$-axis (resp.\ $y$-axis). To this end, we fix two sets of weighted points: $\mathcal{C}_X$ on the positive $x$-axis and $\mathcal{C}_Y$ on the positive $y$-axis. More precisely, an element of $\mathcal{C}_X$ is of the form $((x,0) , s) \in (\RR_+\times \{0 \}) \times \RR$. Similarly, an element $\mathcal{C}_Y$ is of the form $((0,y) , s) \in ( \{0 \} \times \RR_+) \times \RR$. The two sets $\mathcal{C}_X$ and $\mathcal{C}_Y$ can be taken randomly. In order to avoid degeneracy, we will always assume that 
\begin{equation}\label{eq:LF}
\text{the sets of points in $\mathcal{C}_X$ and $\mathcal{C}_Y$ are locally finite a.s.,} \tag{LF} 
\end{equation}
i.e.\ there is no accumulation point on either axis.\par

We also take a PPP $\Xi_0$ on $(0,\infty)^2 \times \RR$ with intensity $\lambda_0 \,\di x \,\di y \,F(0,\di s)$. From the initial conditions $\mathcal{C}_X$ and $\mathcal{C}_Y$, and the parameters $(\lambda_0,\lambda_V,\lambda_H,p_0,p_V,p_H,\tau_V,\tau_H,F)$, we construct a system of lines with the following rules:
\begin{enumerate}
    \item[$1_V$.] From each element $((x,0) ,s) \in \mathcal{C}_X$, we start a vertical line from the point $(x,0)$ going up with weight $s$.
    \item[$1_H$.] From each element $((0,y) ,s) \in \mathcal{C}_Y$, we start an horizontal line from the point $(0,y)$ going right with weight $s$.
    \item[$1_0$.] From each element $((x,y) ,s) \in \Xi_0$, we start an horizontal line from the point $(x,y)$ going right with weight $s$ and a vertical line going up with weight $-s$.
\end{enumerate}

\begin{figure}
   \begin{center}
   \subfloat[{A realization of a PKS.}]{\includegraphics[width=0.4 \textwidth]{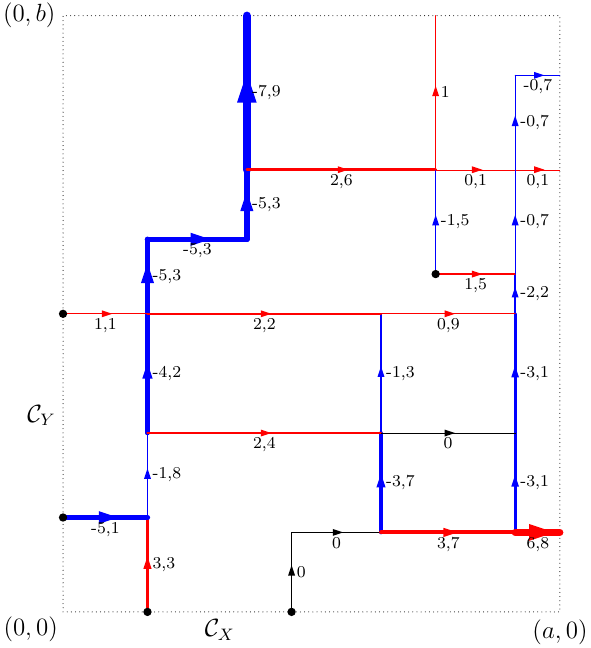}}
   \qquad
   \subfloat[{Simulation of a PKS on $[0,50] \times [0,50]$ according to Model~\ref{ex:NormNorm} of Table~\ref{tab:examples} (see Section~\ref{sec:examples}) with $p_V(s) = p_H(s) = 0.4$ and $\tau_V(s) = \tau_H(s) = 0.1$ whose initial condition $(\mathcal{C}_X,\mathcal{C}_Y)$ are given by two independent PPPs.}]{\includegraphics[width=0.42 \textwidth]{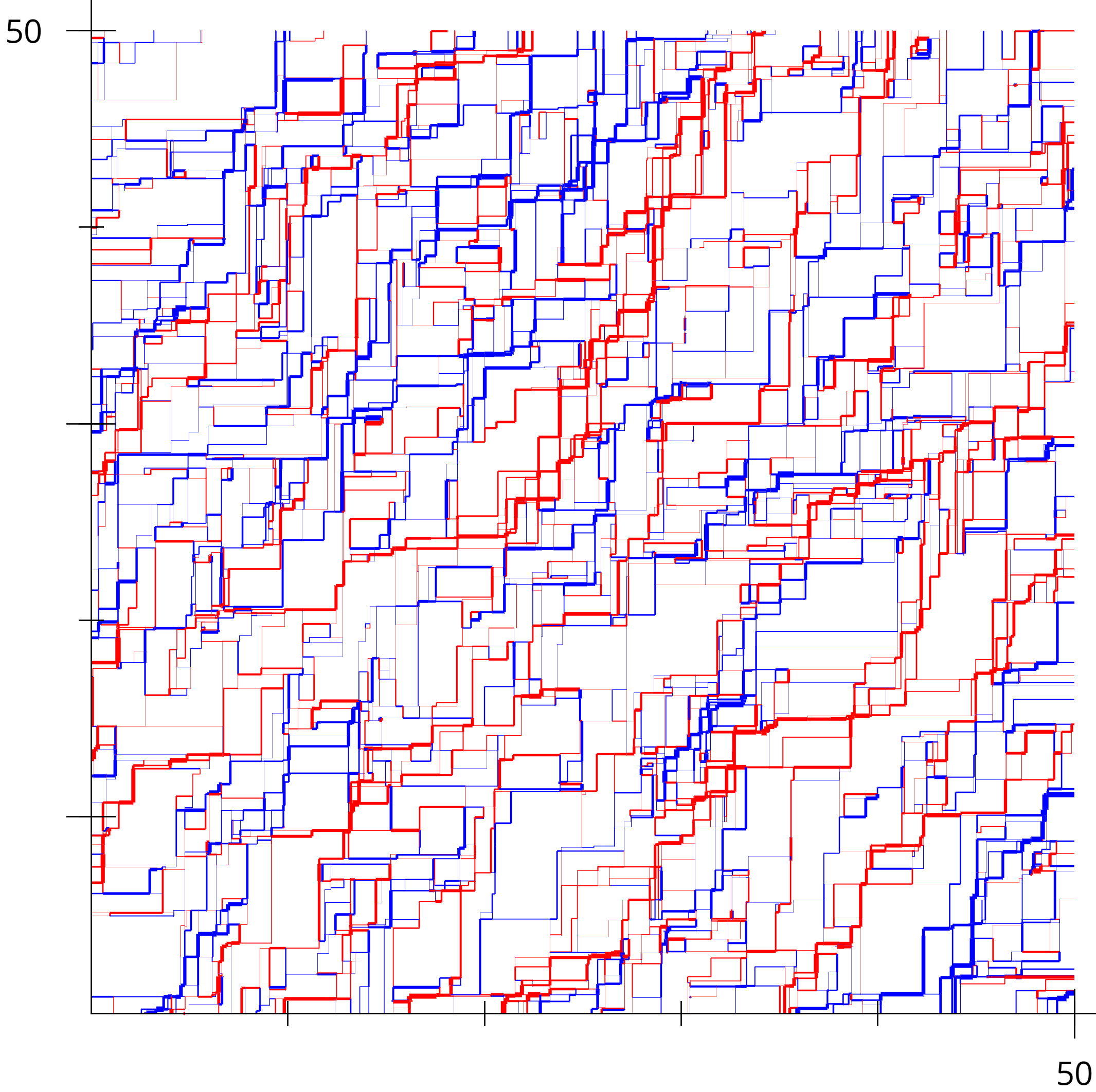}}
   \end{center}
 \caption{Example of dynamics. Lines with positive weights are in red and those with negative weights in blue. The thickness of a line is proportional to the absolute value of its weight.}\label{fig:sim}
\end{figure}

There are two kinds of events that occur during the dynamics.
The first kind concerns what happens to a single line which may \emph{turn} or \emph{split} into two lines. 
\begin{enumerate}
    \item[$2_{V}$.] Along a vertical line of weight $s$:
    \begin{enumerate}    
        \item A split occurs at rate $\lambda_V(s)$. When such an event happens, we pick a random variable $T \sim F(s,\cdot)$, independent of everything else. As a result of this split, the vertical line continues going up with new weight $s-T$, and a horizontal line with weight $T$ starts going right from the point where the split occurs.
        \item The line turns to its right (i.e.\ to the east) at rate $\tau_V(s)$ keeping the same weight and becoming a horizontal line.
    \end{enumerate}
   
    \item[$2_H$.] Along a horizontal line of weight $s$:
    \begin{enumerate}    
        \item  A split occurs at rate $\lambda_H(s)$. When such an event happens, we pick a random variable $T \sim F(s,\cdot)$, independent of everything else. As a result of this split, the horizontal line continues going right with new weight $T$, and a vertical line with weight $s-T$ starts going up from the point where the split occurs.
        \item The line turns to its left (i.e.\ to the north) at rate $\tau_H(s)$ keeping the same weight and becoming a vertical line.
    \end{enumerate}
\end{enumerate}

The second kind of event corresponds to intersections of lines (which we shall refer to as \emph{crossing} events) when a horizontal line going right (i.e.\ coming from the west) with weight $s_W$ meets a vertical line going up (i.e.\ coming from the south) with weight $s_S$. We apply the following rules:
\begin{enumerate}
    \item[$3$.] 
       \begin{enumerate}
        \item with probability $p_V(s_S+s_W)$, the horizontal line stops and the vertical line continues with weight $s_N :=s_S+s_W$;
        \item with probability $p_H(s_S+s_W)$, the vertical line stops and the horizontal line continues with weight $s_E :=s_S+s_W$;
        \item with probability $p_0 \ind{s_S+s_W=0}$, both lines disappear;
        \item on the complementary event, which happens with probability $1-p_V(s_S+s_W) - p_H(s_S+s_W) - p_0 \ind{s_S+s_W=0}$, we pick a random variable $T \sim F(s_S+s_W,\cdot)$, independent of everything else. Then, after meeting each other, the weight of the horizontal line becomes $s_E :=T$ and the weight of the vertical one becomes $s_N :=s_S+s_W-T$. 
    \end{enumerate} 
\end{enumerate}

Rules $1$, $2$ and $3$ together with the initial set of weighted starting points $\mathcal{C}_X$ and $\mathcal{C}_Y$ define a random system of algebraic weighted lines which we call \textbf{Poisson-Kirchhoff System} (PKS) with parameters $(\lambda_0,\lambda_V,\lambda_H,p_0,p_V,p_H,\tau_V,\tau_H,F)$ under the initial condition $(\mathcal{C}_X,\mathcal{C}_Y)$. Let us note that, according to the rules of the dynamics, the system is conservative: it satisfies Kirchhoff's node law (as in Figure~\ref{fig:Kirchhoff}) at every intersection, be it a split, a turn or a crossing. An illustration of a PKS process is given in Figure~\ref{fig:sim}.\par 

\paragraph{Are PKS well defined?} Without further assumptions, the process constructed with the above procedure could be not well defined on the whole quarter plane. Indeed, the previous construction can fail (i.e.\ blow up) if an accumulation of lines appears and prevents us from defining the process any further. From now on, we will say that the PKS is \emph{well defined} if, a.s., the construction above has no accumulation points on the whole quarter plane (or equivalently, there is only a finite number of lines intersecting any bounded region a.s.). The following trivial example illustrates the problem. 

\begin{example} Set $\lambda_0=\lambda_V(s) = \tau_V(s) = \tau_H(s) = p_0=p_V(s) = p_H(s) = 0$. Set $\lambda_H(s) = s^2$ and $F(s,\cdot) = \dirac{s+1}$. Fix $\mathcal{C}_X = \varnothing$ and $\mathcal{C}_Y = \{((0,1), 1)\}$. Then, the PKS starts from of a single horizontal line beginning at point $(0,1)$ on the $y$-axis and with initial weight $1$. This horizontal line never disappears and splits infinitely many times, creating at each split a new vertical line with weight $-1$ while its own weight increases by $1$. Thus, the splitting rate of the horizontal line is equal to $(n+1)^2$ after the $n$th split. This means that the $x$ coordinate of the $n$-th split is equal to $\sum_{i=1}^{n} \frac{\xi_i}{i^2}$ where the $(\xi_i)$ are i.i.d.\ exponential random variables with mean $1$.  The previous sum converges a.s.\ which shows that the PKS blows up almost surely. 
\end{example}

Deciding whether a generic PKS is well defined seems tricky. However, the following elementary result ensures that the PKS is well defined a.s.\ whenever its jump rates are bounded. Later on, the main results in Section~\ref{sec:reversPKS} will provide examples of well-defined PKS with unbounded jump rates.


\begin{proposition} \label{prop:wdFinite}
Assume that $\mathcal{C}_X$ and $\mathcal{C}_Y$ satisfy assumption \eqref{eq:LF}, and that
\begin{equation}\label{eq:boundedrates}
\sup_{s\in \RR}(\lambda_V(s), \lambda_H(s), \tau_V(s), \tau_H(s)) < \infty.
\end{equation}
Then, the PKS is well defined on the whole quarter plane $[0,\infty)^2$ a.s..
\end{proposition}

\begin{proof}
Let us first note that 
\begin{multline*}
\{ \hbox{the PKS is well defined on the whole quarter plane}\} \\
= \bigcap_{a , b \in \NN} \{ \hbox{the PKS is well defined inside the rectangle $[0,a]\times[0,b]$}\}    
\end{multline*}

Thus, we just need to prove that the PKS does not blow up inside any box $[0,a]\times [0,b]$ a.s.. Let us fix such a box $[0,a]\times [0,b]$. Let $y_1 < y_2 < \ldots < y_{N_0}$ denote the $y$-coordinates of the points in $\mathcal{C}_Y$ located on the segment $\{0\}\times[0,b]$. We just need to prove that the PKS is well defined a.s.\ inside $[0,a]\times [0,y_1]$ and then we can repeat the same argument, starting now from height $y_1$, and conclude, after $N_0$ steps that the process is a.s.\ well defined on the whole box.\par
\smallskip
Let $r := \sup_{s\in \RR}(\lambda_V(s), \lambda_H(s), \tau_V(s), \tau_H(s)) < \infty$ and $M_0$ denote the number of weighted points of $\mathcal{C}_X$ located on the segment $[0,a]\times \{0\}$. We follow the dynamics starting from the bottom side of the box and moving upward.\par
Initially, we start with $M_0$ vertical lines going upward. The first split/turn event occurs at some random height $H_1$ which is stochastically larger than an exponential random variable with mean $1/(2M_0r)$ (since all rates are bounded by $r$). At height $H_1$, a new horizontal line is created. This line creates $U_1$ new vertical lines (by splitting and at most one by turning) that will grow upward, and stops $V_1 \geq 0$ vertical lines coming from the bottom (including itself in case of a turn event), see Figure~\ref{fig:dynamic}. Hence, after height $H_1$, the process continues to grow upward with $M_1 = M_0 + U_1 - V_1$ vertical lines. Similarly, after the $n$th split/turn event that occurs at height $H_n$, the process grows up with $M_n = M_{n-1} + U_n - V_n$ vertical lines.\par

\begin{figure}
    \begin{center}
    \includegraphics[scale=0.8]{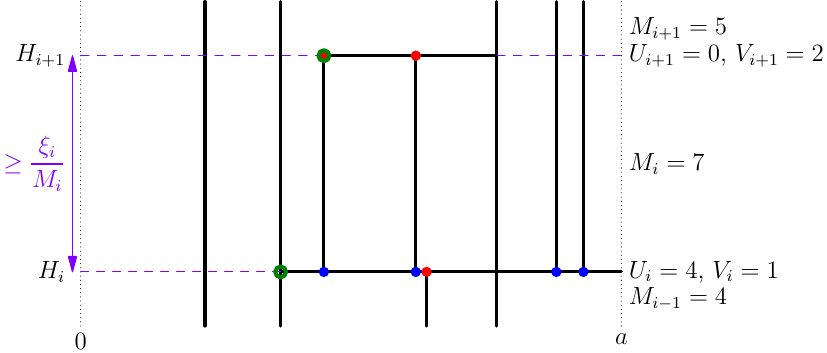}
    \caption{In this example, the $i$th event (green circle) occurs at height $H_i$ and is a split. The ($i$+1)th event (also green circle) occurs at height $H_{i+1}$ and is a turn. On the $i$th event, the horizontal line creates $U_{i}=4$ new lines (blue dots) and stops $V_i=1$ line (red dot). The ($i$+1)th event stops $1$ line additionally to the one that turns (the two red dots) and creates no line. The time between the $i$th event and the ($i$+1)th event is greater than $\xi_i/M_i$ that is an exponential random variable with mean $1/(2 M_i r)$.}
    \label{fig:dynamic}
    \end{center}
\end{figure}

Now, the height $H_n$ of the $n$th split/turn event is stochastically larger than
\begin{equation} \label{eq:sum}
    \sum_{i=0}^{n-1} \frac{\xi_i}{M_i},
\end{equation}
where $(\xi_i)$ are i.i.d.\ exponential random variables with mean $1/(2r)$ which are independent of $(M_i)$. But, remark that since the width of the box is equal to $a$, and that the split and turn rate functions are bounded by $r$, the sequence $(U_i)_{i \geq 1}$ is stochastically dominated by a sequence $(W_i)_{i \geq 1}$ of independent Poisson random variables with parameter $2 a r$. This implies that the sequence of variables $M_n$, which are individually bounded by $M_0 + \sum_{i=1}^n W_i$, grows at most linearly with $n$, and so the sum given in~\eqref{eq:sum} goes a.s.\ to infinity. Hence, the PKS process cannot blow up before reaching height $y_1$, as requested.
\end{proof}

\section{Reversible Poisson-Kirchhoff systems}\label{sec:reversPKS}

 Although the PKS is defined on the whole quarter plane, it is convenient to consider its restriction to a box of the form $[0,a]\times[0,b]$ for $a, b \in \RR^*_+$. We denote by $\dessin_{a,b}$ the image space of the PKS process restricted to the box $[0,a]\times[0,b]$. An element $D\in \dessin_{a,b}$ is called a \emph{drawing} in the box $[0,a] \times [0,b]$. It consists of a finite collection of \emph{weighted vertical} and \emph{horizontal segments} inside this rectangle and which furthermore satisfy the Kirchhoff node law at every intersection (in the sense of Figure~\ref{fig:Kirchhoff}).  

Given a drawing $D \in \dessin_{a,b}$, we define its reverse drawing $\rot{D} \in \dessin_{a,b}$, obtained by rotating $D$ by 180 degrees around the center point $(a/2,b/2)$. Let us note that this rotation yields a valid drawing. An example of a drawing $D$ and its reverse $\rot{D}$ is given in Figure~\ref{fig:reverse}. From now on, we shall denote by  $\vad_{a,b}$ (or simply $\vad$ when the box considered is obvious) a random drawing which has the law of the PKS process defined in Section~\ref{sec:PKS}.

\begin{definition}[Reversibility] \label{def:rev}
A PKS is said to be \emph{reversible} if there exists a random initial condition $(\mathcal{C}_{X},\mathcal{C}_Y)$ such that for any $a,b$, $\vad_{a,b} \eqd \rot{\vad}_{a,b}$.
\end{definition}

\begin{figure}
    \begin{center}
    \begin{tabular}{cc}
        \includegraphics[width = 0.45 \textwidth]{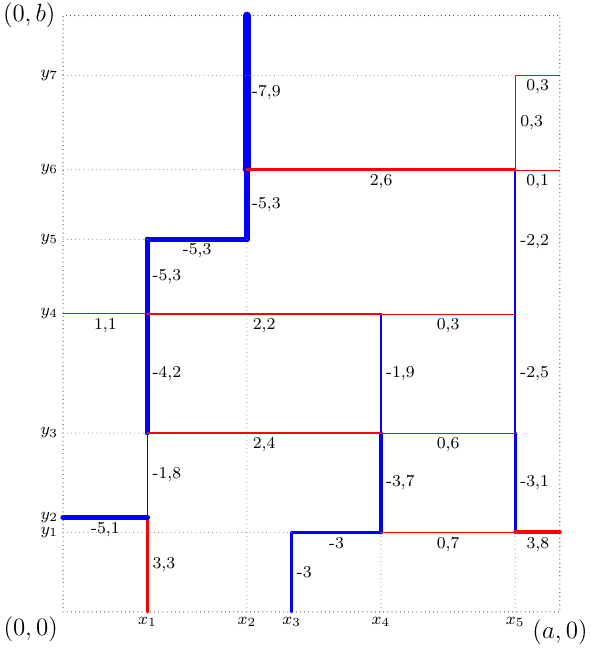} & 
        \includegraphics[width = 0.45 \textwidth]{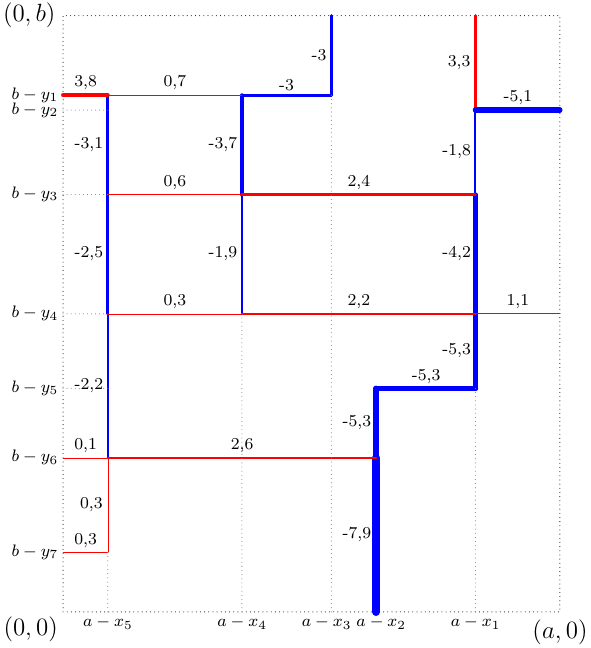}
    \end{tabular}
    \caption{An example of a drawing and, on its right, its reverse}
    \label{fig:reverse}
    \end{center}
\end{figure}

Another related notion is that of \emph{stationarity} of the PKS, which will be implied by the reversibility property in all the cases we shall consider.

\begin{definition}[Stationarity] \label{def:stat}
A PKS is said to be \emph{stationary} if there exists a random initial condition $(\mathcal{C}_{X},\mathcal{C}_Y)$ such that it is translation invariant. Equivalently, this means that the law of a drawing does not depend on the position of the box inside the quarter plane but only on its size. In that case, the law of the initial condition $(\mathcal{C}_X,\mathcal{C}_Y)$ is called an \emph{invariant probability measure} for the PKS.
 \end{definition}\medskip

We now give sufficient conditions on the parameters of a PKS to be reversible under initial conditions ($\mathcal{C}_X$, $\mathcal{C}_Y$) taken as independent weighted PPPs. Thus, from now on, we consider two non-zero finite measures $\nu_V$ and $\nu_H$ on $\RR$ and we will always assume that  
\begin{equation}\label{eq:CXCY-g}
    \begin{cases}
    \text{$\mathcal{C}_X$ is a PPP on $(\RR_+\times \{0 \}) \times \RR$ with intensity $\di x\, \di \nu_V(s)$.} \\
    \text{$\mathcal{C}_Y$ is a PPP on $(\{0 \} \times \RR_+) \times \RR$, independent of $\mathcal{C}_X$, and with intensity $\di y\, \di \nu_H(s)$.}
    \end{cases}
\end{equation}
Let us stress that $\nu_V$ and $\nu_H$ are not necessarily probability measures. In particular, the positions of the vertical lines starting from the $x$-axis is a PPP with intensity $\nu_V(\RR) \di x$. Similarly the positions of the horizontal lines starting from the $y$-axis is a PPP with intensity $\nu_H(\RR) \di y$.

 We call the measures $\nu_V$ (resp.\ $\nu_H$) the \emph{vertical} (resp.\ \emph{horizontal}) \emph{line weight measures}. For technical reasons, we distinguish the following three cases depending on their properties:
\begin{itemize}
    \item when $\nu_V$ and $\nu_H$ are both absolutely continuous w.r.t.\ the Lebesgue measure (Section~\ref{sec:PKSL});
    \item when $\nu_V$ and $\nu_H$ are discrete measures with support included in $\ZZ$ (Section~\ref{sec:discrete});
    \item finally, we discuss the extension of the previous results to arbitrary measures (Section~\ref{sec:gen}).
\end{itemize}

\subsection{Lebesgue case}\label{sec:PKSL}
We assume here that the line weight measures $\nu_V$ and $\nu_H$ are two non-zero finite measures on $\RR$ with Lebesgue densities $g_V$ and $g_H$. Thus, the initial conditions \eqref{eq:CXCY-g} now take the form:
\begin{equation}\label{eq:CXCY}
    \begin{cases}
    \text{$\mathcal{C}_X$ is a PPP on $(\RR_+ \times \{0 \}) \times \RR$ with intensity $\di x\, g_V(s) \di s$.} \\
    \text{$\mathcal{C}_Y$ is a PPP on $(\{0\} \times \RR_+) \times \RR$, independent of $\mathcal{C}_X$, and with intensity $\di y\, g_H(s) \di s$.}
    \end{cases}
\end{equation}
Consider a PKS with parameters $(\lambda_0,\lambda_V,\lambda_H,\allowbreak p_0, p_V, p_H,\allowbreak \tau_V,\tau_H,F)$. The following assumptions ensure the existence of a reversible measure for the PKS process: 

\begin{itemize}
\item[\textbf{(L1)}] The \emph{spontaneous creation rate} is zero, i.e.\ 
\begin{equation} \label{eq:lambda0}
    \lambda_0 = 0.
\end{equation}
Indeed, since we are here in a continuous setting,  case $3$(c) of the dynamics (in Section~\ref{sec:PKS}) never occurs so lines never annihilate. Therefore, in order for the system to be reversible, there must be no spontaneous creation of lines. Consequently, the \emph{annihilation probability} $p_0$ also does not matter here (and can be taken to be zero).\par

\item[\textbf{(L2)}] The \emph{coalescence probability functions} $p_V$ and $p_H$ satisfy the two following conditions with respect to the support of the measures $\nu_V$ and $\nu_H$: for any $s \in \RR$, we have
\begin{equation} \label{eq:notpos}
g_V(s) = 0 \; \Rightarrow \; p_V(s) = 0 \ \text{ and } \ g_H(s) = 0 \;\Rightarrow \; p_H(s) = 0.
\end{equation}

\item[\textbf{(L3)}] The two \emph{turn rate functions} $\tau_V$ and $\tau_H$ satisfy, for any $s \in \RR$, 
\begin{equation}\label{eq:tauRev}
    \tau_V(s) g_V(s) = \tau_H(s) g_H(s).
\end{equation}

\item[\textbf{(L4)}] The two \emph{splitting rate functions} $\lambda_V$ and $\lambda_H$ satisfy, for any $s \in \RR$,
    \begin{equation} \label{eq:lambdaRev}
        \lambda_V(s) = p_V(s) \frac{\GVH(s)}{g_V(s)} \ \text{ and } \  \lambda_H(s) = p_H(s) \frac{\GVH(s)}{g_H(s)},
    \end{equation}
    where $h$ is defined by
    \begin{equation*} 
        \GVH(s) := (g_V \ast g_H)(s) = \int_{\RR} g_V(s-t) g_H(t) \,\di t.
    \end{equation*}
        
\item[\textbf{(L5)}] The \emph{division kernel} $F$ satisfies that, for any $s \in \RR$, the measure $F(s,\cdot)$ is absolutely continuous with respect to the Lebesgue measure and its density $f(s,\cdot)$ is such that, for any $t \in \RR$,
    \begin{equation} \label{eq:fRev}
        f(s,t) =\frac{g_V(s-t)g_H(t)}{\GVH(s)}
    \end{equation}   
    provided that $h(s) > 0$. If $h(s) = 0$, then $f(s,\cdot)$ can be any probability density\footnote{The distribution of $f(s,\cdot)$ when $h(s) = 0$ do not matter. This is just to insure that $f$ is well defined for all $s \in \RR$.} on $\RR$. Notice that this density has a simple probabilistic interpretation: let $X_V$ and $X_H$ be two independent random variables with density proportional to $g_V$ and $g_H$ respectively, then $f(s,\cdot)$ is the density of the variable $X_H$ conditionally on the event $\{X_V + X_H = s\}$.
\end{itemize}

We can now state our main result.

\begin{theorem}[Reversibility in Lebesgue case]\label{thm:reversible}
Consider a PKS with parameters $(\lambda_0,\lambda_V,\lambda_H,\allowbreak p_0, p_V, p_H,\allowbreak \tau_V,\tau_H,F)$. If there exist two non-zero finite measures $\nu_V$ and $\nu_H$ on $\RR$ with densities (according to Lebesgue measure) $g_V$ and $g_H$ such that the previous conditions~\emph{\textbf{(L1)}}, \emph{\textbf{(L2)}}, \emph{\textbf{(L3)}}, \emph{\textbf{(L4)}} and \emph{\textbf{(L5)}} hold, then this PKS under the initial condition $(\mathcal{C}_X,\mathcal{C}_Y)$ as defined in equation~\eqref{eq:CXCY} is well defined and reversible in the sense of Definition~\ref{def:rev}.
\end{theorem}

\begin{remark} \label{rmk:InvMes}
A PKS satisfying the conditions of Theorem~\ref{thm:reversible} may have several reversible measures. Indeed, suppose that there exists a positive constant $r$ such that $\tilde{g}_V(s) = r^s g_V(s)$ and $\tilde{g}_H(s) = r^s g_H(s)$ are still the densities of finite measures. Then Theorem~\ref{thm:reversible} still applies by replacing $g_V$ and $g_H$ by $\tilde{g}_V$ and $\tilde{g}_H$. Consequently, the PKS admits another reversible distribution, given by the law of two independent PPPs: one with intensity $\leb \otimes \tilde{\nu}_V$ and the other one with intensity $\leb \otimes \tilde{\nu}_H$. In that case, the PKS admits a family of reversible measures parameterised by $r$ in an open subinterval of $\RR_+$.
\end{remark}

This theorem is proved in Section~\ref{sec:proofL}. As stated in the next corollary, the reversibility in this case implies the stationarity. Moreover, it also characterizes the law of the restriction of the process along any fixed decreasing curve.

\begin{corollary}\label{cor:revers} Under the conditions of Theorem~\ref{thm:reversible}, the PKS is stationary as defined in Definition~\ref{def:stat}. In particular, the following properties hold:
\begin{enumerate}[label=(\roman*)]

    \item \label{en:stat}  One of its invariant distribution is the law of two independent PPPs $(\mathcal{C}_X, \mathcal{C}_Y)$ where $\mathcal{C}_X$ has intensity $\leb \otimes \nu_V$ and $\mathcal{C}_Y$ has intensity $\leb \otimes \nu_H$.
    \item \label{en:BrokLine} Let $L$ be any broken line of $[0,a] \times [0,b]$ consisting only of eastern and southern steps. Then, $\vad$ restricted to $L$ on its eastern steps is a $\leb\otimes\nu_V$-PPP, and $\vad$ restricted to $L$ on its southern steps is a $\leb\otimes\nu_H$-PPP. These two PPPs are independent.
    \item \label{en:segment} Let $L$ be any straight segment of $\RR^2$: $y=-\alpha x+\beta$ with $x \in [c,d]$ and $\alpha \in (0,+\infty)$. The restriction of $\vad$ to its vertical (resp.\ horizontal) lines is a $\leb\otimes\frac{1}{\sqrt{1+\alpha^2}} \nu_V$-PPP (resp.\ $\leb\otimes\frac{\alpha}{\sqrt{1+\alpha^2}} \nu_H$-PPP). Moreover, these two PPPs are independent.
\end{enumerate}
\end{corollary}

\begin{remark} \label{rmk:RectCurv}
The last result~$\ref{en:segment}$ can still be generalised to any rectifiable curve $\gamma(t) = (x(t), y(t))$ from $[0,1]$ to $[0,a]\times [0,b]$ which is ``decreasing'' in the sense that $x' \geq 0$ and $y' \leq 0$. Then, again, the restrictions of the horizontal and vertical lines of the random drawing to this curve form independent inhomogeneous PPPs whose respective intensities with respect to $\di \lambda \otimes\nu_V$ and $\di \lambda \otimes\nu_H$, where $\di \lambda$ denotes the length measure on the curve, at the point of parameter $t$, are given by the formulas of Corollary~\ref{cor:revers}~$\ref{en:segment}$ taking $\alpha = - \dfrac{y'(t)}{x'(t)}$.
\end{remark}

\begin{proof}[Proof of Corollary~\ref{cor:revers}]
We prove that the restriction of the process to any box $[x,x+a] \times [y,y+b]$ has the same distribution as the one to the box $[0,a] \times [0,b]$, by showing that they are similarly distributed on their left and down boundaries. For this purpose, we first apply Theorem~\ref{thm:reversible} to the box $[0,x+a] \times [0,y]$ and then to the box $[0,x]  \times [y,y+b]$ as illustrated below.

\begin{center}
\includegraphics{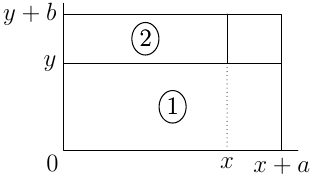}
\end{center}

Indeed, the first application implies that the restriction of the process to the segment $[0,x+a] \times \{y\}$ is distributed as a $\leb\otimes\nu_V$-PPP. Moreover, it is independent of the restriction of the process to the segment $\{0\} \times [y,y+b]$.

Consequently, the second application applies and permits to prove that the restriction of the process to the segment $\{x\} \times [y,y+b]$ is distributed as a $\leb\otimes\nu_H$-PPP. Moreover, this restriction is independent of the one to the segment $[x,x+a] \times \{y\}$ since it only depends on the restrictions to the segments $\{0\} \times [y,y+b]$ and $[0,x] \times \{y\}$ as well as on the dynamics of the process above $y$.
\end{proof}

Another nice consequence of Theorem \ref{thm:reversible} is that the reversibility property makes it straightforward to extend the stationary PKS process defined on the quarter plane $\RR_+^2$ to a stationary process defined on the full plane $\RR^2$. There are several ways to do so. For instance, we can use \ref{en:segment} of Corollary \ref{cor:revers} and start by choosing initially two independents $\leb\otimes\frac{1}{\sqrt{2}} \nu_V$-PPP and $\leb\otimes\frac{1}{\sqrt{2}} \nu_H$-PPP on the anti-diagonal line $y = -x$. We start weighted vertical lines from the atoms of the first PPP (with lines propagating in both top and bottom direction). Similarly, we start weighted horizontal lines from the atoms of the second PPP (with lines propagating in both left and right direction). Then, conditionally on these initial lines, we construct independent processes on the upper region $x > -y$ and the lower region $x > -y$ following the PKS dynamic (\emph{c.f.} Figure \ref{fig:plane}). The resulting process defined on the whole plane $\RR^2$ is translation invariant and its restriction to any box (or quarter plane) coincide with the reversible PKS defined above.

\begin{figure}[h]
    \begin{center}
    \includegraphics[width=5cm]{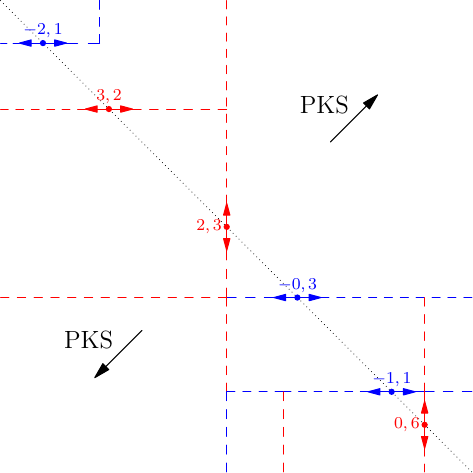}
    \caption{Schematic construction of a reversible PKS process on the full plane.} \label{fig:plane}
    \end{center}
\end{figure}

\subsection{Discrete case} \label{sec:discrete}
We call \emph{discrete case} when all the line weights are integers. Let us start by noticing that, contrary to the Lebesgue setting, case $3$(c) of the dynamics can now occur since two lines with exactly opposite weights can meet. This happens for example with Hammersley's broken line process (schematized in Figure~\ref{fig:CG-fleche}), see~\cite{Hammersley72,AD95,Groeneboom02,CG05,CG06} for additional details. Therefore, in this section, the value of $p_0$ matters and $\lambda_0$ may be non-zero.

Let $\nu_V$ and $\nu_H$ be two non-zero finite measures taking values in $\ZZ$. We consider the initial conditions:
\begin{equation}\label{eq:CXCY-d}
    \begin{cases}
    \text{$\mathcal{C}_X$ is a PPP on $(\RR_+ \times \{0 \}) \times \ZZ$ with intensity $\di x\, \nu_V(\di s)$.} \\
    \text{$\mathcal{C}_Y$ is a PPP on $(\{0 \} \times \RR_+) \times \ZZ$, independent of $\mathcal{C}_X$, and with intensity $\di y\, \nu_H(\di s)$.}
    \end{cases}
\end{equation}

We consider a PKS with parameters $(\lambda_0,\lambda_V,\lambda_H,\allowbreak p_0, p_V, p_H,\allowbreak \tau_V,\tau_H,F)$. the following conditions ensure the reversibility of the process. 
\begin{itemize}
\item[\textbf{(D1)}]
The \emph{spontaneous creation rate} $\lambda_0$ is related to the \emph{annihilation probability} $p_0$ as follows:
    \begin{equation} \label{eq:lambda0-d}
        \lambda_0 = p_0 \sum_{s \in \ZZ} \nu_V(-s) \nu_H(s).
    \end{equation}

\item[\textbf{(D2)}] The \emph{coalescence probability functions} $p_V$ and $p_H$ satisfy two conditions with respect to $\nu_V$ and $\nu_H$: for any $s \in \ZZ$,
\begin{equation} \label{eq:notpos-d}
\nu_V(s) = 0 \Rightarrow p_V(s) = 0 \ \text{ and } \ \nu_H(s) = 0 \Rightarrow p_H(s) = 0.
\end{equation}

\item[\textbf{(D3)}] The two \emph{turn rate functions} $\tau_V$ and $\tau_H$ satisfy, for any $s \in \ZZ$,
    
    \begin{equation} \label{eq:tauRev-d}
        \tau_V(s) \nu_V(s) = \tau_H(s) \nu_H(s).
    \end{equation}

\item[\textbf{(D4)}] The two \emph{splitting rate functions} $\lambda_V$ and $\lambda_H$ satisfy, for any $s \in \ZZ$,
    \begin{equation} \label{eq:lambdaRev-d}
    \lambda_V(s) := p_V(s) \frac{\GVH(s)}{\nu_V(s)} \ \text{ and } \  \lambda_H(s) := p_H(s) \frac{\GVH(s)}{\nu_H(s)},
    \end{equation}
    where 
    \begin{equation*}
        \GVH(s) = (\nu_V \ast \nu_H)(s) = \sum_{t \in \ZZ} \nu_V(s-t) \nu_H(t).
    \end{equation*}
    
\item[\textbf{(D5)}] The \emph{division kernel} $F$ satisfies, for any $s \in \ZZ$, for any $t \in \ZZ$,
    \begin{equation} \label{eq:fRev-d}
        F(s,t) = \frac{\nu_V(s-t)\nu_H(t)}{\GVH(s)}
    \end{equation}
    provided that $\GVH(s) > 0$. If $h(s) = 0$, then $F(s,\cdot)$ can be chosen to be any probability measure.
\end{itemize}

We can now state the theorem in the discrete case.

\begin{theorem}[Reversibility in the discrete case]\label{thm:reversible-d}
Consider a PKS with parameters $(\lambda_0,\lambda_V,\lambda_H,\allowbreak p_0, p_V, p_H,\allowbreak \tau_V,\tau_H,F)$. If there exist two non-zero finite measures $\nu_V$ and $\nu_H$ on $\ZZ$ such that the previous conditions~\emph{\textbf{(D1)}}, \emph{\textbf{(D2)}}, \emph{\textbf{(D3)}}, \emph{\textbf{(D4)}} and \emph{\textbf{(D5)}} hold, then this PKS under the initial condition $(\mathcal{C}_X,\mathcal{C}_Y)$ as defined in equation~\eqref{eq:CXCY-d} is well defined and reversible.
\end{theorem}

\begin{figure}
   \begin{center}
   \subfloat[{Typical lines of a discrete reversible PKS: pairs lines can spontaneously appear and can also annihilate. Particles of weight $0$ may also exist. }]{\includegraphics[width=0.35 \textwidth]{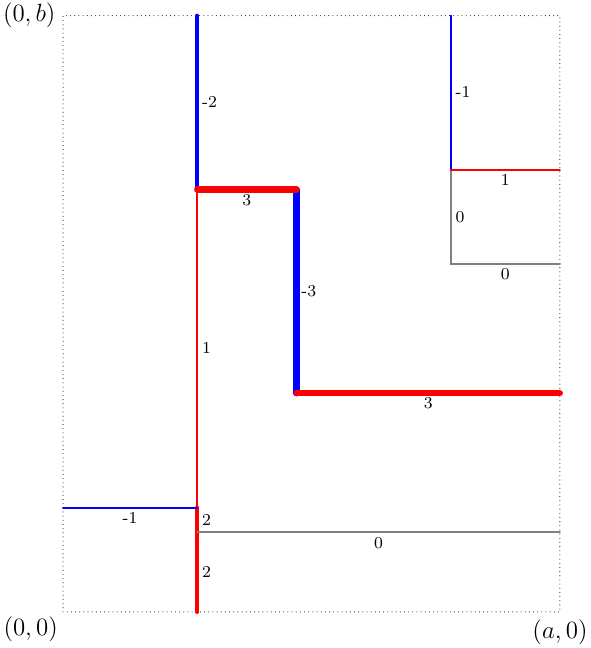}}
   \qquad
   \subfloat[{Simulation of a discrete PKS. In this particular model, we impose that vertical lines are non-positive and horizontal lines are non-negative. Hence, only lines of weight $0$ can turn.
   }]{  \includegraphics[width=0.46\textwidth]{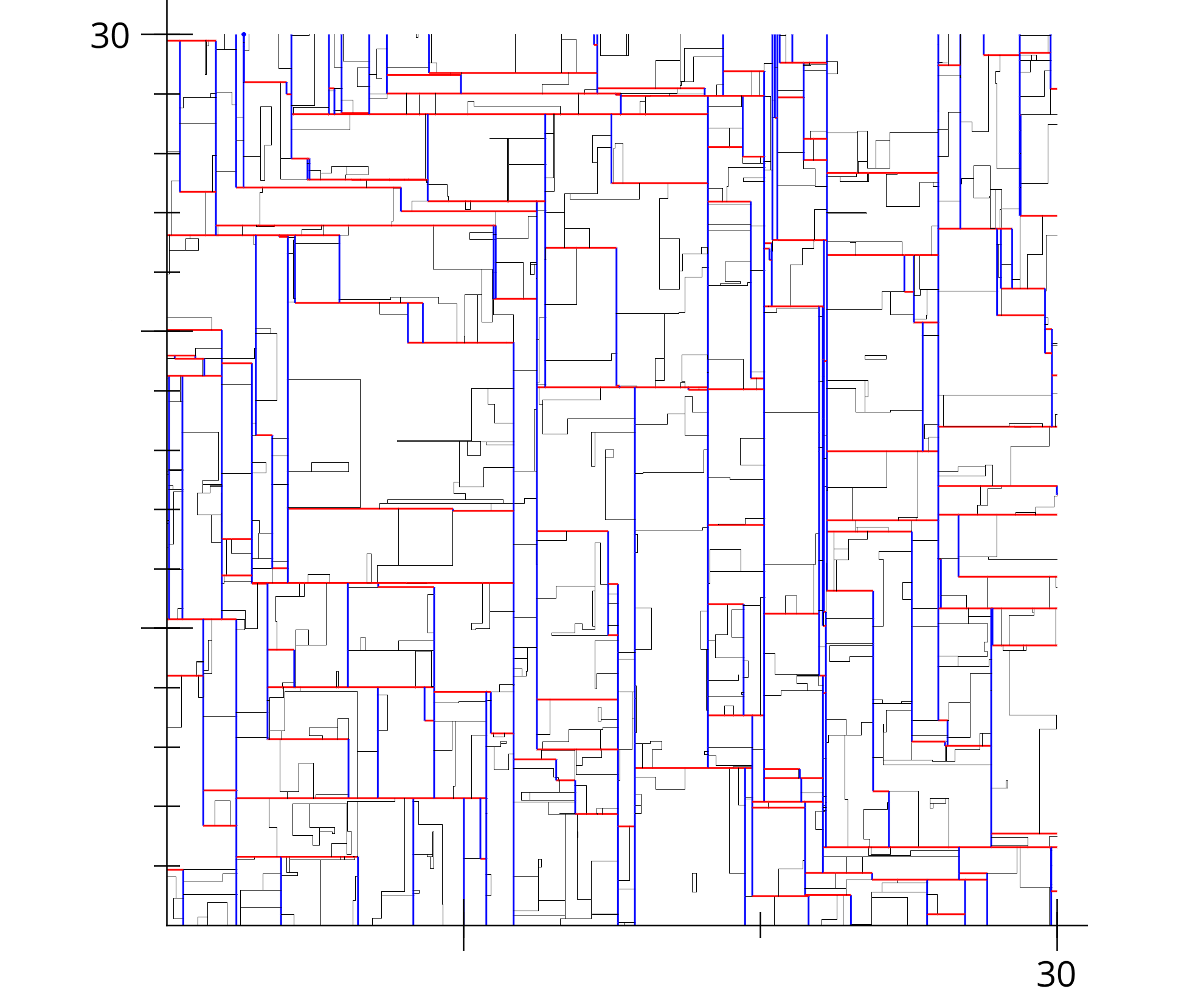}
   \label{fig:sim-discrete_b}}
   \end{center}
   \caption{Examples of a discrete reversible PKS. Lines with positive weight are drawn in red and those with negative weight in blue. Line with weight $0$ are drawn in black.} \label{fig:sim-discrete}
\end{figure}

Corollary~\ref{cor:revers} also holds in this case. 

\begin{remark}
Let us note that Theorems~\ref{thm:reversible} and~\ref{thm:reversible-d} insure that a reversible PKS is well defined when the initial condition is given by one of its invariant measures. It follows that the PKS is also well defined starting from any initial condition that is absolutely continuous w.r.t.\ this invariant measure. In the discrete case, one may check the stronger result that the reversible PKS is, in fact, well defined for any deterministic initial conditions $(\mathcal{C}_X, \mathcal{C}_Y)$. This may not necessarily be true in the Lebesgue case as it is possible to construct forbidden ``pathological'' initial conditions, for example with lines having opposite weights. 
\end{remark}

\subsection{General case}\label{sec:gen}

The Lebesgue and discrete cases described previously represent the most natural settings for PKS. Yet,  it should be possible to generalize the reversibility result to an even more general framework as explained below.

Let us denote by $\mathcal{A}_V$ (resp.\ $\mathcal{A}_H$) the set of atoms of $\nu_V$ (resp.\ $\nu_H$), and set 
\begin{equation*}
\mathcal{A} := \mathcal{A}_H \cap \left(-\mathcal{A}_V\right) = \big\{s \in \RR : \nu_H(\{s\}) \nu_V(\{-s\}) \neq 0\big\}.
\end{equation*}

In the absence of a common measure against which both $\nu_V$ and $\nu_H$ are absolutely continuous, we shall make use of the \RN{} derivatives with respect to $\nu_V$ and to $\nu_H$ in order to define the appropriate rate functions that guarantee reversibility. 

We recall that, according to the \RN{} theorem, given two arbitrary finite measures $\mu$ and $m$ there exists a unique decomposition of $\mu=\mu_{\paral}+\mu_{\perp}$ such that $\mu_{\paral}$ is abs.\ continuous w.r.t.\ $m$ and $\mu_{\perp}$ is singular w.r.t.\ $m$. Henceforth, we define the \RN{} derivative of the measure $\mu$ according to the measure $m$ as $\frac{\di \mu_{\paral}}{\di m}$ and denote it as $\frac{\di \mu}{\di m}$.\par

Consider a PKS with parameters $(\lambda_0,\lambda_V,\lambda_H,\allowbreak p_0, p_V, p_H,\allowbreak \tau_V,\tau_H,F)$. In this general framework, the 5 conditions for reversibility take the form:

\begin{itemize}
 \item[\textbf{(G1)}] The \emph{spontaneous creation rate} $\lambda_0$ satisfies
    \begin{equation}\label{eq:lambda0-g}
    \lambda_0 = p_0 \sum_{s \in \mathcal{A}} \nu_V(\{-s\}) \nu_H(\{s\}).
    \end{equation}
\item[\textbf{(G2)}]  The \emph{coalescence probability functions} $p_V$ and $p_H$ satisfy 
    \begin{equation} \label{eq:notpos-g}
        p_V \in L^{\infty}(\RR,\borel{\RR},\nu_V) \text{ and } p_H \in L^{\infty}(\RR,\borel{\RR},\nu_H).
    \end{equation}

\item[\textbf{(G3)}] The two \emph{turn rate functions} $\tau_V$ and $\tau_H$ satisfy, for any $s \in \RR$, 
    \begin{equation} \label{eq:tauRev-g}
        \tau_V(s) = \tau_H(s)  \frac{\di \nu_H}{\di \nu_V}(s), \ \text{or equivalently \;} \tau_H(s) = \tau_V(s)  \frac{\di \nu_V}{\di \nu_H}(s).
    \end{equation}

\item[\textbf{(G4)}] The two \emph{splitting rate functions} $\lambda_V$ and $\lambda_H$ satisfy, for any $s \in \RR$,
    \begin{equation} \label{eq:lambdaRev-g}
        \lambda_V(s) = p_V(s) \frac{\di\MGVH}{\di \nu_V}(s) \text{ and } \lambda_H(s) = p_H(s) \frac{\di\MGVH}{\di \nu_H}(s),
    \end{equation}
    where $\MGVH = \nu_V \ast \nu_H$ is the convolution product of $\nu_V$ and $\nu_H$, i.e.\ for any $A \in \borel{\RR}$,
    \begin{equation*}
        \MGVH(A) = \int_{\RR^2} \ind{t + s \in A} \, \di \nu_V(t)\, \di \nu_H(s).
    \end{equation*}

\item[\textbf{(G5)}]  The \emph{division kernel} $F$ satisfies, for any $s \in \RR$, for any $t \in \RR$,
    \begin{equation}\label{eq:fRev-g}
        F(s,A) = \int_{A}\frac{\di \nu^{(t)}_V}{\di \MGVH}(s) \, \di \nu_H(t),
    \end{equation}
    where $\nu^{(t)}_V$ the $t$-translated measure\footnote{The $t$-translated measure $\nu^{(t)}_V$ of $\nu_V$ is the measure defined by, for any $A \in \borel{\RR}$, $\nu^{(t)}_V(A) = \nu_V(\{x-t \in \RR : x \in A\})$} of $\nu_V$. The probability kernel $F$ can be also seen as the regular conditional probability of $X \sim \nu_H$ with respect to $\sigma(X+Y)$ where $Y \sim \nu_V$ and $Y$ is independent of $X$ as defined in~\cite[Section~4.1.3]{Durrett19}.\par
\end{itemize}

Then, as in Theorems~\ref{thm:reversible} and~\ref{thm:reversible-d}, a PKS with parameters $(\lambda_0,\lambda_V,\lambda_H,\allowbreak p_0, p_V, p_H,\allowbreak \tau_V,\tau_H,F)$ such that there exist two non-zero finite measures $\nu_V$ and $\nu_H$ on $\RR$ that satisfy the conditions~\textbf{(G1)}, \textbf{(G2)}, \textbf{(G3)}, \textbf{(G4)} and \textbf{(G5)} is well defined and reversible under the initial condition $(\mathcal{C}_X,\mathcal{C}_Y)$ defined in equation~\eqref{eq:CXCY-g}. The proof of the previous statement is similar in spirit but much more technical than for the Lebesgue and discrete cases. The main difficulty being measurability problems stemming from the fact that we do not have access to a reference translation invariant measure against which both weight measures $\nu_V$ and $\nu_H$ are absolutely continuous. However, as mentioned before, in practice, PKS of interest are either discrete or continuous. The full proof of this statement is omitted from the paper (but the scheme of proof and the heuristic given below still apply).

\section{Heuristic}

Conditions~\textbf{(L1-L5)}, \textbf{(D1-D5)} and~\textbf{(G1-G5)} seem technical and somewhat ad hoc at first glance. However, they appear naturally when studying the PKS dynamics at the microscopic scale. Before providing the rigorous (and technical) proof of the main theorems in the next section, we give below a heuristic argument that hopefully shed some light on the necessity of the assumptions.

We look at the PKS process inside an ``infinitely small'' rectangle $\di x \times \di y$ so that at most one event (split, turn, ...) can occur inside this region simultaneously. In order for the PKS process to be reversible, the probability of any elementary event must be equal to the probability of the corresponding event when the  $\di x \times \di y$ rectangle is rotated by 180 degrees. Thus, we can consider, in turn,  each of the 12 possible elementary events pictured in Figure \ref{fig:12config} and check the relations that they entail on the parameters of the process. The first 3 elementary events (empty square, single vertical line and single horizontal line) are symmetric by rotation of 180 degrees so they entail no condition. 

\begin{figure}
    \begin{center}
    \begin{tabular}{c|c||cc||cc} 
    \includegraphics{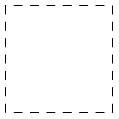} & 
    \includegraphics{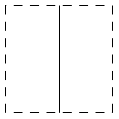} & 
    \includegraphics{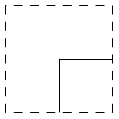} & 
    \includegraphics{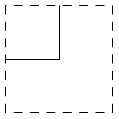} & 
    \includegraphics{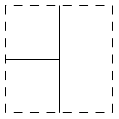} & 
    \includegraphics{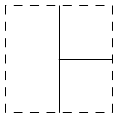} \\
    1 & 2 & 5 & 6 & 9 & 10 \\
    \hline
    \includegraphics{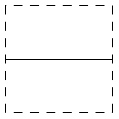} &
    \includegraphics{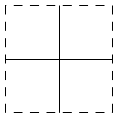} & 
    \includegraphics{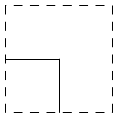} & 
    \includegraphics{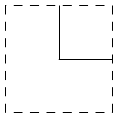} & 
    \includegraphics{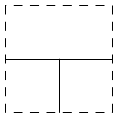} & 
    \includegraphics{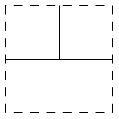} \\
    3 & 4 & 7 & 8 & 11 & 12
    \end{tabular}
    \caption{The twelve local configurations: 4  self-dual configurations (1-4) and 4 pairs (5-12).}
    \label{fig:12config}
    \end{center}    
\end{figure}

\subsection{Crossing}
Let us consider the elementary crossing event 4 of Figure \ref{fig:12config} and its rotation by 180 degrees. 
\begin{center}
    \includegraphics[height=3cm]{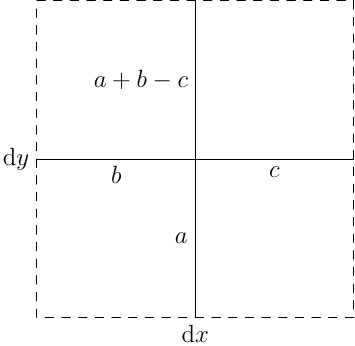} \hspace{3cm} \includegraphics[height=3cm]{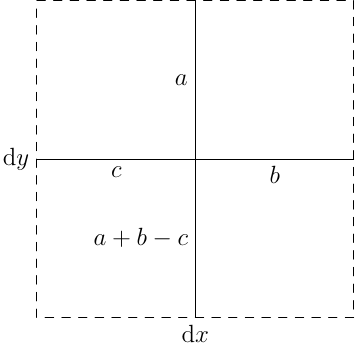} 
\end{center}
In order for reversibility to hold true, both configurations above should appear with equal probability. This means that, given three test functions $u$, $v$ and $w$, the expectation of $u(a) v(b) w(c)$  should be the same on both events. In view of the PKS dynamics, this entails that:
\begin{equation}\label{eq:crosseq_G}
\iiint u(a) v(b) w(c) \di \nu_V(a) \di \nu_H(b) F(a+b,\di c) = \iiint u(a) v(b) w(c) \di \nu_V(a+b-c) \di \nu_H(c) F(a+b,\di b).
\end{equation}
In particular, in the case of discrete measures, the above equation implies the equality
\begin{equation}\label{eq:crosseq_D}
\nu_V(a) \nu_H(b) F(a+b,c) = \nu_V(a+b-c) \nu_H(c) F(a+b,b)
\end{equation}
whereas, in the Lebesgue case, it implies the relation on the densities:
\begin{equation}\label{eq:crosseq_L}
g_V(a) g_H(b) f(a+b,c) = g_V(a+b-c) g_H(c) f(a+b,b).
\end{equation}
It is not difficult to check (excluding possible degenerates cases) that equation \eqref{eq:crosseq_L} (resp.  \eqref{eq:crosseq_D} and \eqref{eq:crosseq_G}) is equivalent to \textbf{(L5)} (resp. \textbf{(D5)} and \textbf{(G5)}).

\subsection{Horizontal turn versus vertical turn}
Let us now look at the complementary events 5 and 6 of Figure \ref{fig:12config}.
\begin{center}
    \includegraphics[height=3cm]{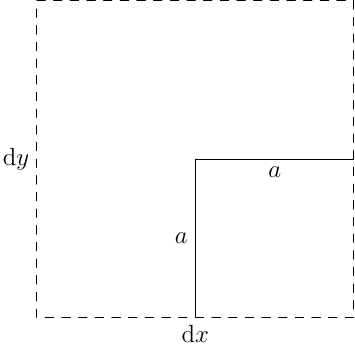} \hspace{2cm} \includegraphics[height=3cm]{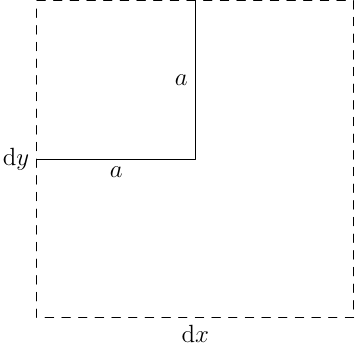}
\end{center}
Just as in the previous section, the reversibility property implies that, for any test functions $u$, it must hold that
\begin{equation*}
    \int u(a) \tau_V(a) \di \nu_V(a) = \int u(a) \tau_H(a) \di \nu_H(a).
\end{equation*}
Therefore, in the discrete case, we have
\begin{equation*}
\nu_V(a) \tau_V(a) = \nu_H(a) \tau_H(a).
\end{equation*}
and in the Lebesgue case
\begin{equation*}
g_V(a) \tau_V(a) = g_H(a) \tau_H(a).
\end{equation*}
The previous three equations are equivalent to \textbf{(G3)}, \textbf{(D3)} and \textbf{(L3)} respectively. 

\subsection{Spontaneous creation versus annihilation}
We consider the complementary events 7 and 8 of Figure \ref{fig:12config}.
\begin{center}
\includegraphics[height=3cm]{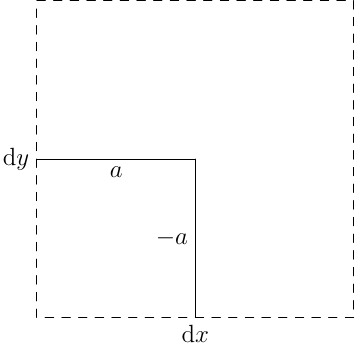} \hspace{2cm} \includegraphics[height=3cm]{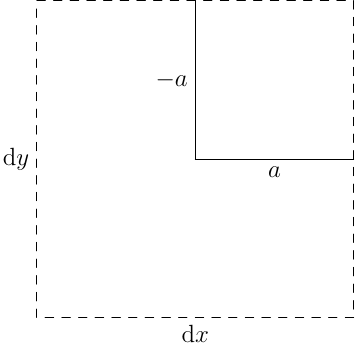}
\end{center}
In order for reversibility to hold true, they should have the same probability. In particular, $a$ in the picture above must be an atom of the measure $\nu_H$ and $-a$ an atom of $\nu_V$. In the Lebesgue case, the probability of the left event is zero which implies that $\lambda_0 = 0$ which is exactly \textbf{(L1)}. In the general case, we find that, for any such atom $a$, we must have
\begin{displaymath}
\nu_V(-a) \nu_H(a) p_0 = \lambda_0 F(0,a)
\end{displaymath}
which is, already assuming \textbf{(G5)} (resp. \textbf{(D5)}), equivalent to \textbf{(G1)} (resp. \textbf{(D1)}).

\subsection{Split versus coalescence}
Let finally consider the two elementary events 9 and 10 of Figure \ref{fig:12config}:  
\begin{center}
    \includegraphics[height=3cm]{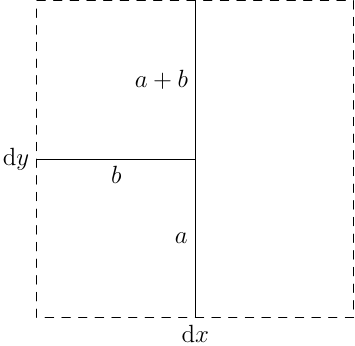} \hspace{2cm} \includegraphics[height=3cm]{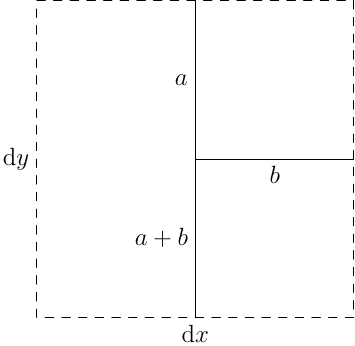}
\end{center}
Using the same argument as before, we now find that, for any two test functions $u$ and $v$, it must hold that
\begin{equation}
\iint u(a) v(b) p_V(a+b) \di \nu_V(a) \di \nu_H(b) = \iint u(a) v(b) \di \nu_V(a+b) \lambda_V(a+b) F(a+b,\di b).
\end{equation}
In the case of discrete measures, the above equation implies the equality
\begin{equation}
\nu_V(a) \nu_H(b) p_V(a+b) = \nu_V(a+b) \lambda_V(a+b) F(a+b,b)
\end{equation}
whereas, in the Lebesgue case, it implies the relation on the densities:
\begin{equation}\label{eq:splicoV}
g_V(a) g_H(b) p_V(a+b) = g_V(a+b) \lambda_V(a+b) f(a+b,b).
\end{equation}
Similarly, considering now events 11 and 12 of Figure \ref{fig:12config}: 
\begin{center}
    \includegraphics[height=3cm]{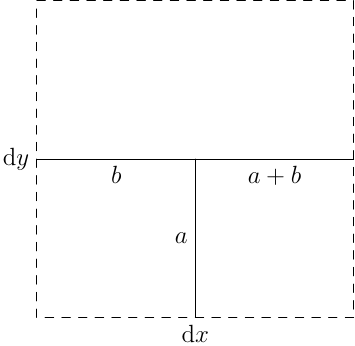} \hspace{2cm} \includegraphics[height=3cm]{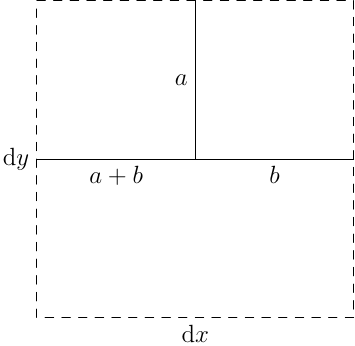}
\end{center}
We find that
\begin{equation}
\iint u(a) v(b) p_H(a+b) \di \nu_V(a) \di \nu_H(b)  = \iint u(a) v(b) \di \nu_H(a+b) \lambda_H(a+b) F(a+b,\di b)
\end{equation}
which translate, in the discrete case to
\begin{equation}
\nu_V(a) \nu_H(b) p_H(a+b) = \nu_H(a+b) \lambda_H(a+b) F(a+b,b) 
\end{equation}
and in the Lebesgue case to
\begin{equation}\label{eq:splicoH}
g_V(a) g_H(b) p_H(a+b) = g_H(a+b) \lambda_H(a+b) f(a+b,b).
\end{equation}
Under the assumption that \textbf{(L5)} holds, then \eqref{eq:splicoV} and \eqref{eq:splicoH} are equivalent to \textbf{(L2)} and \textbf{(L4)}. The same holds true for the discrete and general cases.

\medskip

The previous analysis shows that the five conditions of the previous section are indeed necessary (excluding maybe some degenerate cases) to have reversibility of the PKS process on the microscopic scale. In the next section, we prove that those conditions are actually sufficient and imply, in fact, reversibility on the macroscopic scale. 

\section{Proof of reversibility} \label{sec:proof}
We prove here Theorems~\ref{thm:reversible} and~\ref{thm:reversible-d}. In the next two subsections, we prove the results assuming further that all the rates are uniformly bounded. Next, we show in the last subsection that we can bootstrap the results from the bounded rate case to the general case by using an approximation procedure of a arbitrary PKS by a sequence of PKS with bounded rates. 

\subsection{Proof of Theorem~\ref{thm:reversible} with uniformly bounded rates} \label{sec:proofL}
Consider a PKS in the Lebesgue case with parameters $(\lambda_0,\lambda_V,\lambda_H,\allowbreak p_0, p_V, p_H,\allowbreak \tau_V,\tau_H,F)$, such that there exist two non-zero finite measures on $\RR$ with densities $g_V$ and $g_H$ such that the conditions~\eqref{eq:lambda0}, \eqref{eq:notpos}, \eqref{eq:tauRev}, \eqref{eq:lambdaRev} and~\eqref{eq:fRev} hold. We start this PKS process with  the initial condition $(\mathcal{C}_X,\mathcal{C}_Y)$ as defined in equation~\eqref{eq:CXCY}.\par

In this section, we assume that the rates of the PKS are uniformly bounded, that is to say they satisfy condition~\eqref{eq:boundedrates}, which implies that the PKS is well defined a.s.\ by Proposition~\ref{prop:wdFinite}. Hence, we just need to show the reversibility of the PKS. The uniformly bounded rates assumption will be relaxed in Section~\ref{sec:ProofExists}.\par \bigskip

Recalling the definition of a drawing and of reversibility in Section~\ref{sec:reversPKS}, we want to prove that, for any non-negative measurable function $\Phi : \dessin_{a,b} \to \RR_+$,
\begin{equation} \label{eq:reverse}
\esp{\Phi(\vad)} = \esp{\Phi(\rot{\vad})}
\end{equation}
which exactly states that $\vad$ and $\rot{\vad}$ have the same law. However, the set of all drawings which is infinite dimensional is not a very convenient space to work with. To overcome this difficulty, we partition the set of drawings according to their combinatorial nature which will enable us to rewrite the expectation above as a sum of expectations over finite dimensional spaces. Before doing so, we introduce some notation and definitions that will be helpful to understand the combinatorial structure of a drawing.\par
It will be convenient to represent a \emph{weighted vertical} (resp.\ \emph{horizontal}) \emph{segment} $\sigma$ as a triplet $(\sigma_-,\sigma_+,s)$ where the endpoints are $\sigma_- = (x,y_-)$ (resp.\ $(x_-,y)$) and $\sigma_+ = (x,y_+)$ (resp.\ $(x_+,y)$) with $y_- < y_+$ (resp.\ $x_- < x_+$) and the weight is  $s\in \RR$. 

\paragraph{Types of nodes.} We can define eleven types of nodes that correspond to events in the dynamics occurring inside the box as well as events on the boundary of the domain. For each type of node, we introduce a notation as a pictogram for the set of all nodes of this type.
\begin{itemize}
   \item {\bf Vertical entry}: a vertical entry is a boundary point $(x,0)$ on the bottom side of the box which has an outgoing segment $\sigma \in D$, $\sigma = ((x,0),(x,.),.)$. We denote this set by $\VE$. Remark that $\VE = \{(x,0) : \exists s \in \RR , ((x,0),s) \in \mathcal{C}_X\}$.
    \item {\bf Vertical exit}: a vertical exit is a boundary point $(x,b)$ on the top side of the box which has an outgoing segment  $\sigma \in D$, $\sigma = ((x,.),(x,b),.)$. We denote this set by $\VS$.    
    \item {\bf Vertical split}: a vertical split of $D$ is a point $(x,y)$ where 3 segments are meeting from the south, north and east, i.e.\ there exist $\sigma_S, \sigma_N, \sigma_E \in D$ such that $\sigma_{S} = ((x,,),(x,y),.)$, $\sigma_{N} = ((x,y),(x,.),.)$ and $\sigma_{E} = ((x,y),(.,y),.)$. This corresponds to case $2_V$(a) in the dynamics defined in Section~\ref{sec:PKS}. We denote this set by $\HB$.    
    \item {\bf Vertical turn}: a vertical turn of $D$ is a point $(x,y)$ where 2 segments are meeting from the south and east, i.e.\ there exist $\sigma_S, \sigma_E \in D$ such that $\sigma_{S} = ((x,.),(x,y),.)$ and $\sigma_{E} = ((x,y),(.,y),.)$. This corresponds to case $2_V$(b) in the dynamics. We denote this set by $\HT$.    
    \item {\bf Vertical coalescence}: a vertical coalescence is a point  $(x,y)$ where 3 segments are meeting from the west, south and north, i.e.\ there exist $\sigma_W, \sigma_S, \sigma_N \in D$ such that $\sigma_{W} = ((.,y),(x,y),.)$, $\sigma_{S} = ((x,.),(x,y),.)$ and $\sigma_{N} = ((x,y),(x,.),.)$. This corresponds to case $3$(a) in the dynamics. We denote this set by $\HA$.
\end{itemize}

For all these kinds of nodes, we also define their obvious horizontal counterpart: {\bf horizontal entry} $\HE$, {\bf horizontal exit} $\HS$, {\bf horizontal split} $\VB$, {\bf horizontal turn} $\VT$ and {\bf horizontal coalescence} $\VA$. Finally, we define a last kind of nodes:

\begin{itemize}
    \item {\bf Crossing}: a crossing is a point  $(x,y)$ where 4 segments are meeting. This corresponds to case $3$(d) of the dynamics. Alternatively, this event can be interpreted as a coalescence immediately followed by a split. We denote this set by $\CC$.
\end{itemize}

\paragraph{Skeleton and parametrization of a drawing.} We introduce the notion of \emph{skeleton} of a drawing which will be instrumental in the rest of the proof. We say that two drawings $D, D' \in \dessin$ have the same skeleton, and denote it by $D \sim D'$, if there exist two increasing functions $\psi_X$ from $[0,a]$ to $[0,a]$ and $\psi_Y$ from $[0,b]$ to $[0,b]$ such that for any weighted segment $\sigma = ((x_-,y_-),(x_+,y_+),s) \in D$, there exists a unique $s' \in \RR$ such that $\psi(\sigma) := \big((\psi_X(x_-),\psi_Y(y_-)),(\psi_X(x_+),\psi_Y(y_+)),s'\big) \in D'$.\par
In other words, the skeleton represents the ``combinatorial'' structure of a drawing where we forget about the exact positions and weights of segments, so that two drawings with the same skeleton can be mapped from one to the other by changes of space and weight. Thus, two drawings with the same skeleton $S$ have the same numbers of segments $\l = \l(S)$ as well as the same number of nodes of each type. Furthermore, a skeleton induces a graph whose edges will be denoted $(e_1 , \ldots, e_\ell)$ (for some arbitrary ordering) in the following. An illustration of a drawing and its skeleton is given on Figure~\ref{fig:skeleton}.\par \medskip

\begin{figure}
    \begin{center}
    \begin{tabular}{cc}
    \includegraphics[width = 0.4 \textwidth]{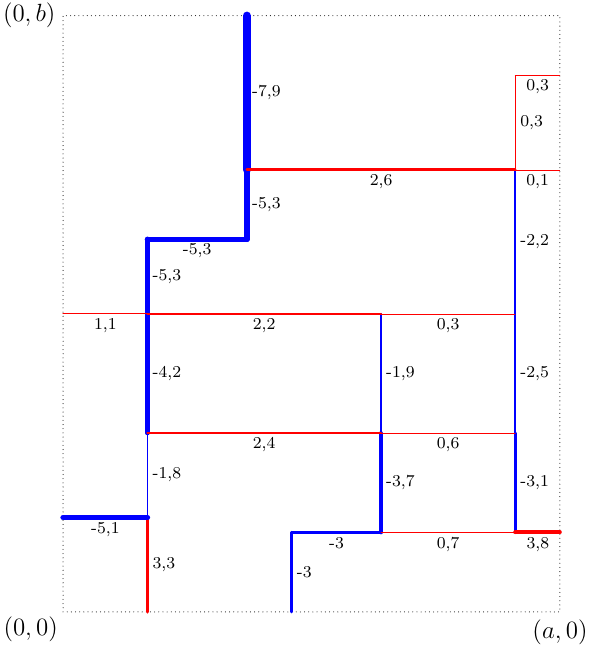} & \includegraphics[width = 0.4 \textwidth]{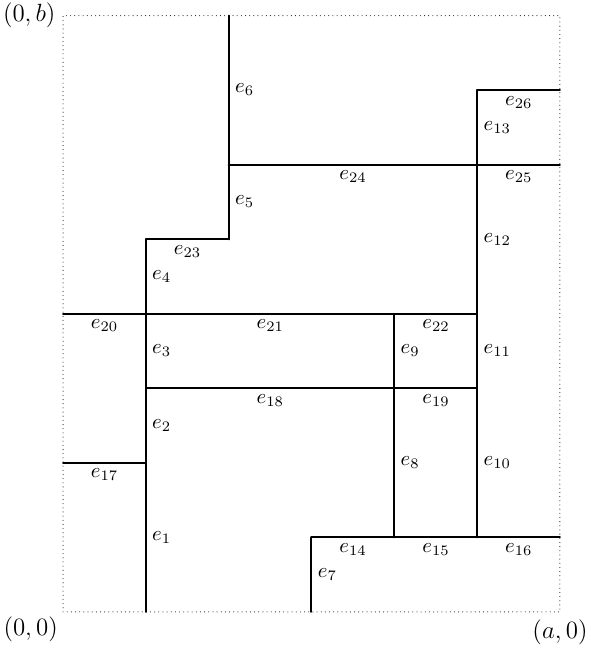}
    \end{tabular}
    \caption{An example of a drawing and, on its right, its skeleton. In this example, $\l=26$.}
    \label{fig:skeleton}
    \end{center}
\end{figure}

Let us note that a drawing $D$, given its skeleton $S$, is uniquely determined once we specify the spatial positions of its segments together with their weights. Thus, we shall now identify the set of all drawings $D$ with skeleton $S$ as a subset of $\RR^{m+n+\ell}$ and we shall represent a drawing $D$ by a vector 
\begin{equation} \label{eq:corres}
(x_1,\ldots, x_m, y_1,\ldots, y_{n}, s_1,\ldots, s_\ell) \in (0,a)^{m} \times (0,b)^{n} \times \RR^{\l}
\end{equation}
where the $(x_i)$'s are the $m := |\VE|+|\VB|+|\VT|$ horizontal coordinates of the points in $\VE \cup \VB \cup \VT$ ordered increasingly, the $(y_i)$'s are the $n := |\HE| + |\HB|+|\HT|$ vertical coordinates of the points in $\HE \cup \HB \cup \HT$ ordered increasingly and the $(s_i)$'s are the weights of the segments corresponding to the edges $(e_i)$ of the skeleton $S$.

However, not all such vectors represent a valid drawing since the Kirchhoff's node law induces relations between segment weights, so the dimension of the space generated by all valid vectors is smaller than $m + n + \l$. More precisely, its dimension is $m + n + d$ where 
\begin{equation}\label{eq:defd}
   d = d(S) =  \l - (|\HB|+|\VB|+|\HT| +|\VT| + |\HA| +|\VA|+|\CC|). 
\end{equation} 
Indeed, we notice that each internal node (i.e.\ a node belonging to $\HB \cup \VB \cup \HT \cup \VT \cup \HA \cup \VA \cup \CC$) adds an independent linear constraint, coming from Kirchhoff's node law, which decreases the space dimension by $1$. We can now derive the following lemma:
\begin{lemma}\label{lem:nskelet}
For any skeleton $S$, the dimension of the set of admissible weights of a drawing $D$ with a given skeleton $S$ is equal to: 
\[ d(S)= |\VE|+|\VB|+|\HE|+|\HB|+|\CC|.\]
\end{lemma}

\begin{proof}
 By counting the number of half-edges of $S$, which is equal to $2\l(S)$, we get:
    \[ 2\l(S) = \left(|\VE| + |\VS| + |\HE| + |\HS|\right) + 2\left(|\HT| + |\VT| \right) + 3\left( |\VB|  + |\VA| + |\HB| + |\HA| \right) + 4 |\CC|.\]
    Indeed, each node in $\VE \cup \VS \cup \HE \cup \HS$ contributes for $1$ half-edge, each node of $\HT \cup \VT$ for $2$ half-edges, each node in $\VB \cup \VA \cup \HB \cup \HA$ for $3$ half-edges and each node in $\CC$ for $4$ half-edges. 
        
 Moreover, remark that
 \[ \left\{
 \begin{aligned}
 |\VE| + |\VB| + |\VT| & = |\VS| + |\VA| + |\HT| \text{\quad (because both are equal to $m$)},\\
 |\HE| + |\HB| + |\HT| & = |\HS| + |\HA| + |\VT| \text{\quad (because both are equal to $n$)}.
 \end{aligned}
 \right.
\]
 Consequently,
 \[    2\l(S)  = 2 \left( |\VE| + |\HE| + |\HT| + |\VT| + 2 |\VB| + |\VA| + 2 |\HB| + |\HA| + 2 |\CC|\right).
 \]
   Then, by using equation~\eqref{eq:defd},
   \begin{align*}
   d(S) & =  \left( |\VE| + |\HE| + |\HT| + |\VT| + 2 |\VB| + |\VA| + 2 |\HB| + |\HA| + 2 |\CC| \right) \\
    & \quad - (|\HB|+|\VB|+|\HT| +|\VT| + |\HA| +|\VA|+|\CC|) \\
    & = |\VE| + |\HE| + |\HB| + |\VB| +|\CC|. \qedhere
   \end{align*}
\end{proof}

Define a \emph{parametrization} of a skeleton $S$ by selecting $d$ edges $(e_{\rho(1)}, \ldots, e_{\rho(d)})$ where $\rho$ is an injective mapping from $\{1,\dots, d\}$ to $\{1,\dots,\ell\}$ such that the knowledge of the weights on the edges $e_{\rho(1)}, \ldots, e_{\rho(d)}$ together with Kirchhoff's node law entirely defines the weights of all edges in the skeleton. 
In particular, a parametrization defines an injective linear mapping  $\mathfrak{D}_{S,\rho}: \RR^{d(S)} \to \RR^{\l(S)}$ whose image is the vector space generated by valid drawing vectors (i.e.\ satisfying Kirchhoff's law at each node), and where the $j$th coordinate corresponds to the weight $s_{\rho(j)}$ on the edge $e_{\rho(j)}$, i.e.\
\begin{equation} 
    \mathfrak{D}_{S,\rho}((c_j)_{j=1..d}) = (s_i)_{i=1..\l}
\end{equation}
with $s_{\rho(j)} = c_{j}$ for any $j$.

\paragraph{A parametrization related to the dynamics.} A particular parametrization $\rho_S$ related to the dynamics of the PKS defined in Section~\ref{sec:PKS} is obtained by selecting only the vertical (resp.\ horizontal) edges whose starting point belongs to $\VE$ (resp.\ $\HE \cup \VB \cup \HB \cup \CC$). In term of the dynamics, this means that we keep track of the weights of the entry points and of the weights of the eastern edges when split or crossing events occur.

It is clear that this subset of edges yields a valid parametrization of a drawing since this family has the correct cardinal $d(S)$ and since all weights in the drawing can be reconstructed iteratively by following the dynamics of the process. See Figure \ref{fig:coord-dyn} for an illustration. 

\begin{figure}
    \begin{center}
    \begin{tabular}{cc}
        \includegraphics[width = 0.4 \textwidth]{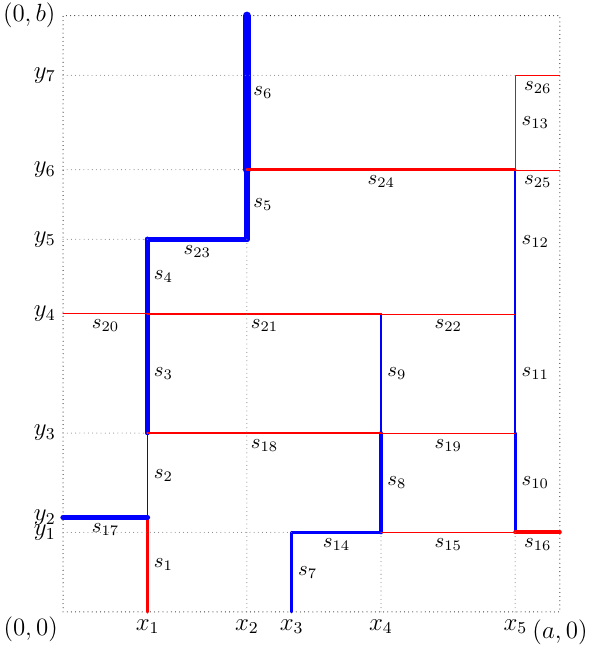} & \includegraphics[width = 0.4 \textwidth]{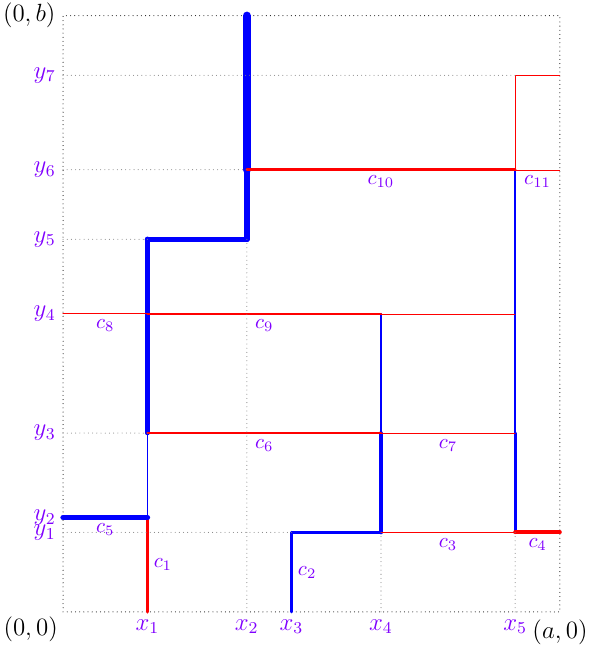}
    \end{tabular}
    \caption{On the left, a drawing $D$ with all its coordinates in $\RR^{m+n+\l}$ and, on the right, the same drawing with its free $m+n+d$ coordinates chosen as in Section~\ref{sec:proofL}.  In this example, $\l=26$, $m=5$, $n=7$ and $d=11$.} 
    \label{fig:coord-dyn}
    \end{center}
\end{figure}

Using the parametrization $\rho_S$, we can decompose the expectation $\espp{}{\Phi(\vad)}$ in equation~\eqref{eq:reverse} with the following formula:
\begin{equation} \label{eq:Phi-can}
\espp{}{\Phi(\vad)} = \sum_{S \in \dessin / \sim} \int_{\RR^{m+n+d}} \di x\ \di y\ \di{c}\  \Phi\big((x,y,\mathfrak{D}_{S,\rho_S}(c))\big)\ \alpha_{S} \big((x,y,\mathfrak{D}_{S,\rho_S}(c)) \big),
\end{equation}
where $\alpha_{S}$ is to be thought of as the ``density'' of the drawing $\vad$ on the event that its skeleton is $S$ (and when using the parametrization $\rho_S$ described previously).

Before expressing $\alpha_S$, we introduce  the function $q:\RR \to \RR_+$, we will refer as the \emph{turn function} defined, for any $s \in \RR$, by
    \begin{equation*} 
    q(s) := \tau_V(s) \sqrt{\frac{g_V(s)}{g_H(s)}} = \tau_H(s) \sqrt{\frac{g_H(s)}{g_V(s)}} \text{\quad by equation~\eqref{eq:tauRev},}
    \end{equation*}
    with the convention $0/0 = 0$ in the formula above. Beyond simplifying the expression of $\alpha_{S}$, the introduction of the additional function $q$ will be of great help in the proof of the invariance of $\alpha_{S}$ by the reverse operation $\rot{\cdot}$ (further Lemma~\ref{lem:rotalpha}), since it will turn out to be itself invariant by this operation.

\begin{lemma} \label{lem:alpha} For any skeleton $S$, and any drawing $D$ whose skeleton is $S$ and identified to $$(x_1,\dots,x_m,y_1,\dots,y_n,s_1,\dots,s_\l),$$ see equation~\eqref{eq:corres}, we have
\begin{align*}
     \alpha_{S}(D) =\ & \left( \ind{0< x_1 < x_2 < \dots  < x_{m} < a} \right) e^{-\left(\int_{\RR} g_V(s) \di s\right) a}\, \left( \ind{0< y_1 < y_2< \dots  < y_{n} < b} \right) e^{-\left(\int_{\RR} g_H(s) \di s\right) b} \\
    &\Bigg( \prod_{\sigma = ((x_-,y_-),(x_+,y_+),s) \in D}  \bigg[ \ind{x_-=x_+}\, g_V(s)^{\left( \ind{(x_-,y_-) \in \ttVE \cup \ttVB \cup \ttHB \cup \ttCC} - \ind{(x_+,y_+) \in \ttHB \cup \ttHT} \right)} \\
    & \sqrt{q(s) g_V(s)}^{\ind{(x_-,y_-) \in \ttVT} +\ind{(x_+,y_+) \in \ttHT}} \, p_V(s)^{\left(\ind{(x_-,y_-) \in \ttHA} + \ind{(x_+,y_+) \in \ttHB}\right)} e^{-\left(\tau_V(s) + \lambda_V(t) \right) (y_+-y_-)} \\
    & \phantom{\Bigg( \prod_{\sigma = ((x_-,y_-),(x_+,y_+),s) \in D}}+ \ind{y_-=y_+} \, g_H(s)^{\left(\ind{(x_-,y_-) \in \ttHE \cup \ttHB \cup \ttVB \cup \ttCC} - \ind{(x_+,y_+) \in \ttVB \cup \ttVT}\right)}\\
    &\sqrt{q(s) g_H(s)}^{\ind{(x_-,y_-) \in \ttHT} +\ind{(x_+,y_+) \in \ttVT}} \, p_H(s)^{\left(\ind{(x_-,y_-) \in \ttVA} + \ind{(x_+,y_+) \in \ttVB}\right)} e^{- \left(\tau_H(s) + \lambda_H(t) \right) (x_+-x_-)} \bigg] \Bigg)\\
    & \left( \prod_{(x,y) \in \ttCC} \frac{1-p_V(s_W+s_S)-p_H(s_W+s_S)}{ \GVH(s_W+s_S)}\ \ind{((x,.),(x,y),s_S) \in D}\ \ind{((.,y),(x,y),s_W) \in D}\right).
    \end{align*}
\end{lemma}

\begin{proof}
The formula above is nothing more than a rearrangement of a product of terms where each one represents the probability of a local event which, put together, ensures that $D$ is indeed a drawing with skeleton $S$ chosen according to the Poisson-Kirchhoff dynamics. Let us analyse each term separately.\par
First, the indicator functions $\ind{0< x_1 < \dots < x_{|\ttVE| +|\ttVB| + |\ttVT|}< a}$ and $\ind{0 < y_1 < \dots <y_{|\ttHE| + |\ttHB| + |\ttHT|}< b}$ ensure that the $(x_i)_{i}$'s and $(y_j)_{j}$'s are correctly ordered.\par
Secondly, the terms $e^{-\left(\int_{\RR} g_V(s) \di s\right) a}$ and $e^{-\left(\int_{\RR} g_H(s) \di s\right) b}$ are respectively equal to the probabilities that there is no other entry on the bottom and left boundaries $[0,a] \times \{0\}$ and $\{0\} \times [0,b]$.\par 
Thirdly, each segment $\sigma = ((x_-,y_-), (x_+,y_+),s)\in D$ contributes to the product through the terms 
\begin{displaymath}
 e^{-\left(\tau_V(s) + \lambda_V(t) \right) (y_+-y_-)} \text{ or } e^{-\left(\tau_H(s)+ \lambda_H(t) \right) (x_+-x_-)},
\end{displaymath}  which represents the probability of non-splitting and non-turning along the segment $\sigma$ depending on whether it is vertical or horizontal.

Finally, we look at the contribution to the density of each node $(x,y)$ and show how it can be decomposed into factors associated to each segment adjacent to the node, and to the node itself when it is a crossing, i.e.\ when $(x,y)$ belongs to $\CC$. We distinguish the following cases with respect to the node type, using the notation $s_N, s_E, s_S$ and $s_W$ for the sizes of the northern, eastern, southern and western segments which are adjacent to $(x,y)$:
\begin{itemize}
    \item if $(x,y) \in \VE$, its northern segment $\sigma_N$, which is its only adjacent segment, gets the contribution $g_V(s_N) = g_V(s_N)^{\ind{(\sigma_N)_- \in \ttVE}}$ coming from the vertical entry of the boundary PPP. 
    
    \item if $(x,y) \in \VS$, no contribution is assigned to the southern segment, which is its only adjacent segment, since it is an exit point.
    
    \item if $(x,y) \in \HB$, the term $\lambda_V(s_S) f(s_S,s_E) =  p_V(s_S)\dfrac{g_H(s_E) g_V(s_S-s_E)}{g_V(s_S)}$ splits  into three terms which are distributed  on the three adjacent segments to $(x,y)$ as follows:
    \begin{itemize}
        \item the term $g_H(s_E) = g_H(s_E)^{\ind{(\sigma_{E})_- \in \ttHB}}$ on the eastern segment $\sigma_{E}$,
        \item the term $g_V(s_S - s_E) = g_V(s_N) = g_V(s_N)^{\ind{(\sigma_{N})_- \in \ttHB}}$ on the northern segment $\sigma_{N}$,
        \item the term $p_V(s_S)/g_V(s_S) = p_V(s_S)^{\ind{(\sigma_S)_+ \in \ttHB}} g_V(s_S)^{-\ind{(\sigma_S)_+ \in \ttHB}} $ on the western segment $\sigma_{S}$.
    \end{itemize}
    
    \item if $(x,y) \in \HT$, the term 
    \begin{displaymath}
    \tau_V(s_S) = q(s_S) \sqrt{\frac{g_H(s_S)}{g_V(s_S)}} = \sqrt{q(s_E) g_H(s_E)} \frac{\sqrt{q(s_S) g_V(s_S)}}{g_V(s_S)}
    \end{displaymath} splits into two terms which are distributed on the two segments adjacent to $(x,y)$ as follows:
    \begin{itemize}
        \item the term $\sqrt{q(s_E)g_H(s_E)} = \sqrt{q(s_E)g_H(s_E)}^{\ind{(\sigma_E)_-\in \ttHT}}$ on the eastern segment $\sigma_E$,
        \item the term $\sqrt{q(s_S)g_V(s_S)}/g_V(s_S) = \sqrt{q(s_S) g_V(s_S)}^{\ind{(\sigma_S)_+\in \ttHT}} g_V(s_S)^{-\ind{(\sigma_S)_+\in \ttHT}}$ on the southern segment $\sigma_S$;
    \end{itemize}
    
    \item if $(x,y) \in \HA$, the term $p_H(s_S+s_W)= p_H(s_N) = p_H(s_N)^{\ind{(\sigma_N)_- \in \ttHA}}$ is assigned to its northern segment $\sigma_{N}$. Its southern and western adjacent segments get no contribution.
    
    \item The contributions of nodes of horizontal type $\HE$, $\HS$, $\VB$, $\VT$ and $\VA$ are decomposed analogously as the nodes of vertical type above.
    
    \item if $(x,y) \in \CC$, the term $(1-p_V(s_S+s_W) -p_H(s_S+s_W))\dfrac{g_V(s_N) g_H(s_E)}{\GVH(s_S+s_W)}$ splits into three terms which are distributed as follows:
    \begin{itemize}
        \item the term $g_V(s_N) = g_V(s_N)^{\ind{(\sigma_N)_- \in \ttCC}}$ on the northern segment $\sigma_N$,
        \item the term $g_H(s_E) = g_H(s_E)^{\ind{(\sigma_E)_- \in \ttCC}}$ on the eastern segment $\sigma_E$,
        \item the term $\dfrac{1-p_H(s_S+s_W)-p_H(s_S+s_W)}{\GVH(s_S+s_W)}$ is attached to the node itself,
        \item its southern and western adjacent segments get no contribution. \qedhere
    \end{itemize}
\end{itemize}
\end{proof}

\paragraph{Change of parametrization.} Formula~\eqref{eq:Phi-can} presents a decomposition of the expectation of $\Phi(\vad)$ in terms of the special parametrization $\rho_S$ defined above. However, this formula is in fact valid for any parametrization $\rho$ thanks to the following lemma:

\begin{lemma} \label{lem:det}
Let $\rho$ and $\rho'$ denote two parametrizations of a skeleton $S$ with respective linear mappings $\mathfrak{D}_{S,\rho}$ and $\mathfrak{D}_{S,\rho'}$ from $\RR^d$ to $\RR^\l$. We have
    \begin{equation*}
    \left|\det\left( \mathfrak{D}_{S,\rho}^{-1} \circ \mathfrak{D}_{S,\rho'}\right)\right| = 1.
    \end{equation*} 
\end{lemma}

\begin{proof}
Notice that if $\rho$ and $\rho'$ have the same image, then the application $\mathfrak{D}_{S,\rho}^{-1} \circ \mathfrak{D}_{S,\rho'}$ is just a permutation and the result is trivial. We will prove the lemma in the case where $\rho$ and $\rho'$ differ only by one coordinate. Then, the general case will follow by choosing a finite sequence of parametrizations where two consecutive parametrizations differ by exactly one coordinate. 

Take now $\rho$ and $\rho'$ such that they differ only by one coordinate. Without loss of generality, we can assume that for all $i\leq d-1$, $\rho(i) = \rho'(i)$. Consider the set of edges $e$ in $S$ such that $(e_{\rho(1)},\dots,e_{\rho(d-1)},e)$ is a parametrization of $S$.
This set is necessarily connected. Indeed, if this was not the case then we could pick an edge in each connected component and add it to the parametrization since, according to the Kirchhoff's node law, setting the weight of an edge can only constrain the weight of edges in the same connected component. But this would yield a parametrization with more than $d$ edges, which is absurd.\par
Consequently, there exists a path $(e^{(1)}=\rho(d),\dots, e^{(k)}=\rho'(d))$ such that, for any $i$, $e^{(i)}$ and $e^{(i+1)}$ are adjacent in $S$. Now, for any $i$, let $\mathfrak{D}^{(i)} = (e_{\rho(1)},\dots,e_{\rho(d-1)},e^{(i)})$. Finally, we just need to check that $\left|\det \left( (\mathfrak{D}^{(i)})^{-1}\circ \mathfrak{D}^{(i+1)}\right)\right| = 1$. This is clearly the case because, according to Kirchhoff's law around the node shared by $e^{(i)}$ and $e^{(i+1)}$, we have $s(e^{(i)}) = \pm s(e^{(i+1)}) + \sum_{i=1}^{d-1} \lambda_i s(e_{\rho(i)})$ for some fixed $(\lambda_i)$.
\end{proof}

\begin{corollary} \label{cor:det}
The formula~\eqref{eq:Phi-can} still holds true when replacing $\rho_S$ by any parametrization $\rho$.
\end{corollary}

\begin{proof}
Let $\rho$ be any parametrization of $S$. Doing the change of variable $c' = \mathfrak{D}^{-1}_{S,\rho} \circ \mathfrak{D}_{S,\rho_S}(c)$ in equation~\eqref{eq:Phi-can} and applying Lemma~\ref{lem:det}, we get 
\begin{align*}
\espp{}{\Phi(\vad)} & = \sum_{S \in \dessin / \sim} \int_{\RR^{m+n+d}} \di x\ \di y\ \di{c}\  \Phi\big((x,y,\mathfrak{D}_{S,\rho_S}(c))\big)\ \alpha_{S} \big((x,y,\mathfrak{D}_{S,\rho_S}(c)) \big) \\
& = \sum_{S \in \dessin / \sim} \int_{\RR^{m+n+d}} \di x\ \di y\ \di{c'}\  \Phi\big((x,y,\mathfrak{D}_{S,\rho}(c'))\big)\ \alpha_{S} \left((x,y,\mathfrak{D}_{S,\rho}(c')) \right). \qedhere
\end{align*}
\end{proof}

\paragraph{Reversibility.} The last ingredient we need to prove the reversibility of the model is the invariance of the density $\alpha_{S}$ by the rotation of $180$ degrees. 
\begin{lemma}\label{lem:rotalpha}
  For any skeleton $S$, for any drawing $D$ with skeleton $S$, we have $\alpha_{\rot{S}}(\rot{D}) = \alpha_{S}(D)$.
\end{lemma}

\begin{proof}
The function $\alpha_{S}(D)$ only depends on the length and weight of the segments and on the crossings of the drawing $D$. To any segment $\sigma = ((x_-,y_-),(x_+,y_+),s) \in D$, we associate its reverse segment $\rot{\sigma} = ((\rot{x}_-,\rot{y}_-),(\rot{x}_+,\rot{y}_+),\rot{s}) \in \rot{D}$ where $\rot{x}_- = a - x_+$, $\rot{x}_+ = a - x_-$, $\rot{y}_- = b - y_+$, $\rot{y}_+ = b - y_-$, and $\rot{s} = s$ (obviously, the weight of a segment does not change by a rotation of 180 degrees).\par

In particular, for any $i$, $\rot{x}_i = a - x_{m+1-i}$ and $\rot{y}_i = b - y_{n+1-i}$. Hence, the terms in the first line of the expression of $\alpha_{S}(D)$ in Lemma~\ref{lem:alpha} and its $\alpha_{\rot{S}}(\rot{D})$-counterpart coincide.

Consider now a vertical segment $\sigma$. Its contribution in $\alpha_{S}(D)$ equals 
\begin{align*}
&  g_V(s)^{\left( \ind{(x_-,y_-) \in \ttVE \cup \ttVB \cup \ttHB \cup \ttCC} - \ind{(x_+,y_+) \in \ttHB \cup \ttHT}\right)}\sqrt{q(s)g_V(s)}^{\ind{(x_-,y_-)\in \ttVT} +\ind{(x_+,y_+) \in \ttHT}}\\ 
& \quad p_V(s)^{\left(\ind{(x_-,y_-) \in \ttHA} + \ind{(x_+,y_+) \in \ttHB}\right)}\ e^{-\left(\tau_V(s) + \lambda_V(s) \right) (y_+-y_-)}.
\end{align*}
The contribution of the reverse segment $\rot{\sigma}$ to $\alpha_{\rot{S}}(\rot{D})$ is equal to 
\begin{align*}
 & g_V(\rot{s})^{\left( \ind{(\rot{x}_-,\rot{y}_-) \in \rot{\ttVE} \cup \rot{\ttVB} \cup \rot{\ttHB} \cup \rot{\ttCC}} - \ind{(\rot{x}_+,\rot{y}_+) \in \rot{\ttHB} \cup \rot{\ttHT}} \right)} \sqrt{q(\rot{s}) g_V(\rot{s})}^{\ind{(\rot{x}_-,\rot{y}_-) \in \rot{\ttVT}}+\ind{(\rot{x}_+,\rot{y}_+) \in \rot{\ttHT}}} \\
 & \quad p_V(\rot{s})^{\left(\ind{(\rot{x}_-,\rot{y}_-) \in \rot{\ttHA}} + \ind{(\rot{x}_+,\rot{y}_+) \in \rot{\ttHB}}\right)}\ e^{-\left(\tau_V(\rot{s})+ \lambda_V(\rot{s}) \right) (\rot{y}_+-\rot{y}_-)}.
\end{align*}
Let us show that both contributions are equal. Indeed, their fourth terms are equal because $s = \rot{s}$ and $\rot{y}_+ - \rot{y}_- = y_+ - y_-$. Their third terms are also equal since, by Table~\ref{tab:correspondence},
\begin{displaymath}
    p_V(\rot{s})^{\left(\ind{(\rot{x}_-,\rot{y}_-) \in \rot{\ttHA}} + \ind{(\rot{x}_+,\rot{y}_+) \in \rot{\ttHB}} \right)} = p_V(s)^{\left(\ind{(x_+,y_+) \in \ttHB} + \ind{(x_-,y_-) \in \ttHA} \right)}.
\end{displaymath}

\begin{table}
    \begin{center}
    \begin{tabular}{|c||c|c|c|c|c||c|}
    \hline
    $(x,y) \in D$ & $\VE$ & $\VS$ & $\HB$ & $\HT$ & $\HA$ & $\CC$ \\
    \hline
    $(a-x,b-y) = (\rot{x},\rot{y}) \in \rot{D}$ & $\rot{\VS}$ & $\rot{\VE}$ & $\rot{\HA}$ & $\rot{\VT}$ & $\rot{\HB}$ & $\rot{\CC}$\\
    \hline
    \end{tabular}
    \caption{Correspondence between each type of vertical node, and of crossing nodes as viewed in $D$ or in $\rot{D}$. Similar correspondences hold for horizontal nodes.}
    \label{tab:correspondence}
    \end{center}    
\end{table}

Similarly, their second terms are equal since
\begin{displaymath}
    \sqrt{q(\rot{s})g_V(\rot{s})}^{\ind{(\rot{x}_-,\rot{y}_-) \in \rot{\ttVT}} + \ind{(\rot{x}_+,\rot{y}_+) \in \rot{\ttHT}}} = \sqrt{q(s)g_V(s)}^{\ind{(x_+,y_+) \in \ttHT} + \ind{(x_-,y_-) \in \ttVT}}.
\end{displaymath}
Finally, their first terms are equal since
    \begin{align}
        & \ind{(\rot{x}_-,\rot{y}_-) \in \rot{\tVE} \cup \rot{\tVB} \cup \rot{\tHB} \cup \rot{\tCC}} - \ind{(\rot{x}_+,\rot{y}_+) \in \rot{\tHB} \cup \rot{\tHT}} =\ind{(x_+,y_+) \in \tVS \cup \tVA \cup \tHA \cup \tCC} - \ind{(x_-,y_-) \in \tHA \cup \tVT} \label{eq:indicunderbrace}  \\
        & = \displaystyle \ind{(x_-,y_-) \in \tVE \cup \tVB \cup \tHB \cup \tCC} - \ind{(x_+,y_+) \in \tHB \cup \tHT}  + \underbrace{\ind{(x_+,y_+) \in \tVS \cup \tVA \cup \tHA \cup \tCC \cup \tHB \cup \tHT}}_{=1} - \underbrace{\ind{(x_-,y_-) \in \tHA \cup \tVT \cup \tVE \cup \tVB \cup \tHB \cup \tCC}}_{=1}, \nonumber
    \end{align}
where we used that the set $\VS \cup \VA \cup \HA \cup \CC \cup \HB \cup \HT$ collects all the nodes ending a vertical segment and, similarly, the set $\HA \cup \VT \cup \VE \cup \VB \cup \HB \cup \CC$ collects all the nodes beginning a vertical segment. The same considerations holds for horizontal segments.

Finally, the last terms contributing to $\alpha_{S}$ are those that concern crossings in $\CC$. Let us consider a crossing $(x,y) \in \CC$ whose weights of its adjacent edges are denoted by $s_S$, $s_W$, $s_N$, $s_E$. Its contribution to $\alpha_{S}(D)$ equals
\begin{displaymath}
\frac{1-p_V(s_W+s_S)-p_H(s_W+s_S)}{\GVH(s_W+s_S)}.
\end{displaymath}
Similarly, the contribution of $(\rot{x},\rot{y}) \in \rot{\CC}$ to $\alpha_{\rot{S}}(\rot{D})$, is equal to 
\begin{displaymath}
\frac{1-p_V(\rot{s_W}+\rot{s_S})-p_H(\rot{s_W}+\rot{s_S})}{\GVH(\rot{s_W}+\rot{s_S})}.
\end{displaymath}
But, $\rot{s_S} = s_N$ and $\rot{s_W} = s_E$ and, by Kirchhoff's node law, $s_E + s_N = s_W + s_S$. Hence, both contributions coincide again.
\end{proof}

We can now deduce Theorem~\ref{thm:reversible} when $\lambda_V$, $\lambda_H$, $\tau_V$ and $\tau_H$ are uniformly bounded.

\begin{proof}[Proof of Theorem~\ref{thm:reversible} (uniformly bounded rates)]
Let $\Phi : \dessin_{a,b} \to \RR_+$ be a non-negative measurable function. For any skeleton $S$, let $(e_{\rho(1)},\dots,e_{\rho(d)})$ be a parametrization of $S$. Choose any order on the set of edges of $\rot{S}$. Now, for any $i$, the edge $e_{\rho(i)} \in S$ has a reverse edge in $\rot{S}$ whose index in $\rot{S}$ is denoted by $\rot{\rho}(i)$. The set $(e_{\rot{\rho}(1)},\dots,e_{\rot{\rho}(d)})$ is a parametrization of $\rot{S}$.
In the next formula, for a drawing $D$, we write indifferently $\rotop(D)$ or $\rot{D}$. 
\begin{align*}
    \espp{}{\Phi(\rot{\vad})} & = \sum_{S \in \dessin / \sim} \int_{\RR^{m+n+d}} \di x\ \di y\ \di{c}\  \Phi\big(\rotop(x,y,\mathfrak{D}_{S,\rho}(c))\big)\ \alpha_{S} \big((x,y,\mathfrak{D}_{S,\rho}(c)) \big) \\
    & = \sum_{S \in \dessin / \sim} \int_{\RR^{m+n+d}} \di x\ \di y\ \di{c}\  \Phi\big((\rot{x},\rot{y},\mathfrak{D}_{\rot{S},\rot{\rho}}(c))\big)\ \alpha_{S} \big((x,y,\mathfrak{D}_{S,\rho}(c) )\big).
\end{align*}
Now, we apply the change of variable from $(x,y)$ to $(\rot{x},\rot{y})$. Recalling that $\rot{x}_i = a- x_{m+1-i}$ and $\rot{y}_i = b - y_{n+1-i}$, it follows that the absolute value of the Jacobian is equal to $1$, hence
\begin{align}
    \espp{}{\Phi(\rot{\vad})} & = \sum_{S \in \dessin / \sim} \int_{\RR^{m+n+d}} \di x\ \di y\ \di{c}\  \Phi\big((x,y,\mathfrak{D}_{\rot{S},\rot{\rho}}(c))\big)\ \alpha_{S} \big((\rot{x},\rot{y},\mathfrak{D}_{S,\rho}(c) )\big) \label{eq:p1}\\
    & \qquad \text{(by Lemma~\ref{lem:rotalpha})} \nonumber \\ 
    & =  \sum_{S \in \dessin / \sim} \int_{\RR^{m+n+d}} \di x\ \di y\ \di{c}\  \Phi\big((x,y,\mathfrak{D}_{\rot{S},\rot{\rho}}(c))\big)\ \alpha_{\rot{S}} \big((x,y,\mathfrak{D}_{\rot{S},\rot{\rho}}(c) )\big) \label{eq:p2}\\
    & \qquad \text{(by Corollary~\ref{cor:det})} \nonumber\\
     & =  \sum_{S \in \dessin / \sim} \int_{\RR^{m+n+d}} \di x\ \di y\ \di{c}\  \Phi\big((x,y,\mathfrak{D}_{\rot{S},\rho_{\rot{S}}}(c))\big)\ \alpha_{\rot{S}} \big((x,y,\mathfrak{D}_{\rot{S},\rho_{\rot{S}}}(c) )\big)\label{eq:p3}\\
    & \qquad \text{(by re-indexation of $\rot{S}$ into $S$)} \nonumber \\
    & =  \sum_{S \in \dessin / \sim} \int_{\RR^{m+n+d}} \di x\ \di y\ \di{c}\  \Phi\big((x,y,\mathfrak{D}_{S,\rho_S}(c))\big)\ \alpha_{S} \big((x,y,\mathfrak{D}_{S,\rho_S}(c) )\big) \label{eq:p4}\\
    & = \espp{}{\Phi(\vad)}. \nonumber
    \qedhere
\end{align}
\end{proof}

\subsection{Proof of Theorem~\ref{thm:reversible-d} with uniformly bounded rates}\label{sec:proofD}

As above, we prove the theorem under the assumption that the rates of the PKS are uniformly bounded. The proof is mostly identical to the one of Theorem~\ref{thm:reversible} (the Lebesgue case), so we shall only point out the changes needed to deal with the rules $1_0$ and $3$(c) of the dynamics. In particular, two new types of nodes need to be considered:
\begin{itemize}
    \item {\bf Spontaneous split}: a spontaneous split is a point where 2 segments are meeting coming from the north and east i.e.\ there exist $\sigma_N, \sigma_E \in D$ such that $\sigma_{N} = ((x,y),(x,.),.)$ and $\sigma_{E} = ((x,y),(.,y),.)$. This corresponds to case $1_0$ of the dynamics. We denote this set by $\OB$. We remark that $\OB = \{ (x,y) : \exists s\in \ZZ, ((x,y),s) \in \mathcal{C}_0\}$.
    \item {\bf Double coalescence}: a double coalescence is a point  $(x,y)$ where 2 segments are meeting, coming from the west and south i.e.\ there exist $\sigma_W, \sigma_S \in D$ such that $\sigma_{W} = ((.,y),(x,y),.)$ and $\sigma_{S} = ((x,.),(x,y),.)$. This corresponds to case $3$(c) of the dynamics. We denote this set by $\OA$.
\end{itemize}\par \medskip

The notion of skeleton is the same as the one defined in the previous Section~\ref{sec:proofL}. 
But, now, the number of free horizontal coordinates is $m = |\VE| + |\VB| + |\VT|+ |\OB|$, the number of free vertical coordinates is $n = |\HE| + |\HB| + |\HT| + |\OB|$, and the number of free weight coordinates is 
\begin{displaymath}
    d :=  \l - (|\HB|+|\VB|+|\HT|+|\VT|+|\HA|+|\VA|+ |\CC| + |\OA| + |\OB| ).
\end{displaymath}

We can now derive the following lemma instead of Lemma~\ref{lem:nskelet}:
\begin{lemma} For any skeleton $S$, the dimension of the set of admissible weights of a drawing $D$ with a given skeleton $S$ is equal to:
\[ d(S) = |\VE| + |\VB| + |\HE| + |\HB| + |\CC| + |\OB| - |\OA|.\]
\end{lemma}

\begin{proof}
The proof is the same as the one of Lemma~\ref{lem:nskelet}, except that now the number of half-edges of $S$ is
\begin{align*}
     2 \l(S) & = \left(|\VE| + |\VS| +  |\HE| + |\HS| \right) + 2\left(|\HT| +  |\VT| + |\OB|+|\OA| \right) \\ 
     & \quad +  3\left( |\HB| +  |\HA|  +  |\VB| + |\VA| \right) +  4 |\CC| ,
\end{align*}
and the spatial dimensions expressions give
 \[ \left\{
 \begin{aligned}
 |\VE| + |\VB| + |\VT|+ |\OB| & = |\VS| + |\VA| + |\HT| + |\OA| \text{\quad (because both are equal to $m$)},\\
 |\HE| + |\HB| + |\HT| + |\OB| & = |\HS| + |\HA| + |\VT| + |\OA| \text{\quad (because both are equal to $n$)}. \qedhere
 \end{aligned}
 \right. 
\]
\end{proof}

As before, the set of all drawings $D$ with skeleton $S$ is identified as a subset of $\RR^{m+n+\l}$, and a drawing $D$ is represented by a vector as in equation~\eqref{eq:corres}. A parametrization $\rho$ of a skeleton $S$ is the selection of $d$ edges $(e_{\rho(1)},\dots,e_{\rho(d)})$ that permits to define the weights of all edges.\par \medskip

As before, we define $\alpha_{S}$ as the density of the drawing $D$ on the event that is skeleton is $S$, and the \emph{turn function} $q: \ZZ \to \RR_+$ by the following formula: for any $s \in \ZZ$,
\begin{equation*} 
    q(s):= \tau_V(s)\sqrt{\frac{\nu_V(s)}{\nu_H(s)}} = \tau_H(s)\sqrt{\frac{\nu_H(s)}{\nu_V(s)}},
\end{equation*}
with the convention $0/0 =0$. As in Corollary~\ref{cor:det}, for any parametrization $\rho$,
\begin{equation*}
\espp{}{\Phi(\vad)} = \sum_{S \in \dessin / \sim} \int_{\RR^{m+n}} \di x\ \di y\ \sum_{c \in \ZZ^d}\  \Phi\big((x,y,\mathfrak{D}_{S,\rho}(c))\big)\ \alpha_{S} \big((x,y,\mathfrak{D}_{S,\rho}(c)) \big),
\end{equation*}
where $\alpha_{S}$ is given by the following lemma (instead of Lemma~\ref{lem:alpha}):
\begin{lemma} \label{lem:alphaD} 
For any skeleton $S$, and any drawing $D$ whose skeleton is $S$ and identified to 
\begin{displaymath}
(x_1,\dots,x_m,y_1,\dots,y_n,s_1,\dots,s_\l),
\end{displaymath}
see equation~\eqref{eq:corres}, we have $\alpha_S(D) =$
    \begin{align*}
    & \left( \ind{0< x_1 < x_2 < \dots  < x_{m} < a} \right) e^{-\left(\sum_{s \in \ZZ} \nu_V(s) \di s\right) a} \, \left( \ind{0< y_1 < y_2< \dots  < y_{n} < b} \right) e^{-\left(\sum_{s \in \ZZ} \nu_H(s) \right) b} \, e^{- p_0 \GVH(0) ab} \, p_0^{|\tOB|+|\tOA|} \\
    &\Bigg( \prod_{\sigma = ((x_-,y_-),(x_+,y_+),s) \in D}  \bigg[ \ind{x_-=x_+}\, \nu_V(s)^{\left( \ind{(x_-,y_-) \in \ttVE \cup \ttVB \cup \ttHB \cup \ttCC \diff{\cup \ttOB}} - \ind{(x_+,y_+) \in \ttHB \cup \ttHT} \right)} \\
    & \quad \sqrt{q(s) \nu_V(s)}^{\ind{(x_-,y_-) \in \ttVT} + \ind{(x_+,y_+) \in \ttHT}}  \,p_V(s)^{\left(\ind{(x_-,y_-) \in \ttHA} + \ind{(x_+,y_+) \in \ttHB}\right)} e^{-\left(\tau_V(s) + \lambda_V(s) \right) (y_+-y_-)} \\
    & \phantom{\Bigg( \prod_{\sigma = ((x_-,y_-),(x_+,y_+),s) \in D} } + \ind{y_-=y_+} \, \nu_H(s)^{\left(\ind{(x_-,y_-) \in \ttHE \cup \ttHB \cup \ttVB \cup \ttCC \diff{\cup \ttOB}} - \ind{(x_+,y_+) \in \ttVB \cup \ttVT}\right)}\\
    & \quad \sqrt{q(s)\nu_H(s)}^{\ind{(x_-,y_-) \in \ttHT} + \ind{(x_+,y_+) \in \ttVT}}  \,p_H(s)^{\left(\ind{(x_-,y_-) \in \ttVA} + \ind{(x_+,y_+) \in \ttVB}\right)} e^{-\left(\tau_H(s)+\lambda_H(s)\right) (x_+-x_-)} \bigg] \Bigg)\\
    & \left( \prod_{(x,y) \in \ttCC} \frac{1-p_V(s_W+s_S)-p_H(s_W+s_S)-p_0 \ind{s_W+s_S=0}}{\GVH(s_W+s_S)}\ \ind{((x,.),(x,y),s_S) \in D}\ \ind{((.,y),(x,y),s_W) \in D}\right).
    \end{align*}
\end{lemma}
\begin{proof}
This argument is the same as the one of Lemma~\ref{lem:alpha}, with some additional terms. First, the term $e^{-p_0 \GVH(0) ab}$ is equal to the probability that there is no other spontaneous split in the rectangle $[0,a] \times [0,b]$. As before, we look at the contribution of each node and distribute it to its adjacent edges or to the node itself. We detail what happens for the three new kinds of nodes:
    \begin{itemize}
    \item if $(x,y)\in \OB$, the term $p_0 \nu_V(-s) \nu_H(s)$ splits into three terms:
        \begin{itemize}
        \item the term $\nu_V(-s) = \nu_V(-s)^{\ind{(\sigma_N)_{-} \in \ttOB}}$ is assigned to the northern segment $\sigma_N$,
        \item the term $\nu_H(s)=\nu_H(s)^{\ind{(\sigma_E)_{-} \in \ttOB}}$ is assigned to the eastern segment $\sigma_E$,
        \item the term $p_0$ is attached to the node itself. All of these contributions are found in the term $p_0^{|\tOB|}$ in $\alpha_{S}$.
        \end{itemize}
    \item if $(x,y) \in \OA$, the term $p_0$ is assigned to the node itself. All these contributions are found in the term $p_0^{|\tOA|}$ in $\alpha_{S}(D)$. Its adjacent segments get no contribution.
    \item if $(x,y) \in \CC$ such that $s_S = -s_W$, the term $(1-p_H(0)-p_V(0)-p_0)\dfrac{\nu_V(-s)\nu_H(s)}{\GVH(0)}$ splits into three terms:
        \begin{itemize}            
        \item the term $\nu_V(-s) = \nu_V(-s)^{\ind{(\sigma_N)_{-} \in \ttCC}}$ is assigned to the northern segment $\sigma_N$,
        \item the term $\nu_H(s)=\nu_H(s)^{\ind{(\sigma_E)_{-} \in \ttCC}}$ is assigned to the eastern segment $\sigma_E$,
        \item the term $\dfrac{1-p_V(0)-p_H(0)-p_0}{\GVH(0)}$ is attached to the node itself. All of these contributions are found in the terms $\dfrac{1-p_V(s_W+s_S)-p_H(s_W+s_S)-p_0 \ind{s_W+s_S=0}}{\GVH(s_W+s_S)}$ where $s_S=-s_W$ in $\alpha_{S}$.\qedhere
        \end{itemize}            
    \end{itemize}
\end{proof}      
        
Finally, to end the proof of the invariance by rotation of $180$ degrees, we need to prove that a lemma similar to Lemma~\ref{lem:rotalpha} holds in our new case.
\begin{lemma}
  For any skeleton $S$, for any drawing $D$ with skeleton $S$, we have $\alpha_{\rot{S}}(\rot{D}) = \alpha_{S}(D)$.
\end{lemma}
\begin{proof}This proof is the same as the one of Lemma~\ref{lem:rotalpha}, with some new terms to check. First, notice that $\OB = \rot{\OA}$ and $\OA = \rot{\OB}$. 

Hence, the contribution of the following factor in $\alpha_{S}(D)$
    \begin{displaymath}
    e^{- p_0 \GVH(0)ab} \, p_0^{|\tOB|+|\tOA|} \prod_{(x,y) \in \ttCC} \frac{1-p_V(s_W+s_S)-p_H(s_W+s_S)-p_0 \ind{s_W+s_S=0}}{\GVH(s_W+s_S)}
    \end{displaymath}
coincides with the same in $\alpha_{\rot{S}}(\rot{D})$
    \begin{displaymath}
    e^{-p_0h(0)ab} \, p_0^{|\rot{\tOB}|+|\rot{\tOA}|} \prod_{(\rot{x},\rot{y}) \in \rot{\ttCC}} \frac{1-p_V(\rot{s}_W+\rot{s}_S)-p_H(\rot{s}_W+\rot{s}_S)-p_0 \ind{\rot{s}_W+\rot{s}_S=0}}{\GVH(\rot{s}_W+\rot{s}_S)}.
    \end{displaymath}

The first two terms are obviously equal, the last one is equal for the same reason as in the proof of Lemma~\ref{lem:rotalpha} remarking that if $s_S+s_W = 0$, then $\rot{s}_S + \rot{s}_W = s_N + s_E = 0$.

The last point to see is that the following indicator function has changed
    \begin{displaymath}
        \ind{(x_-,y_-) \in \ttVE \cup \ttVB \cup \ttVT \cup \ttHB \cup \ttCC \cup \ttOB}.
    \end{displaymath}
Nevertheless, that does not change the proof of the equation~\eqref{eq:indicunderbrace} because, now, the set that collects all the ending nodes of a vertical segments is $\VS \cup \VA \cup \HT \cup \HA \cup \CC \cup \HB \cup \OA$, and the one that collects all the beginning nodes of a vertical segments is $\HA \cup \VE \cup \VB \cup \VT \cup \HB \cup \CC \cup\OA$.  
\end{proof}    
    
Finally, the end of the proof of Theorem~\ref{thm:reversible-d} with uniformly bounded rates is the same as the one of Theorem~\ref{thm:reversible} with uniformly bounded rates, since equations~\eqref{eq:p1},~\eqref{eq:p2},~\eqref{eq:p3} and~\eqref{eq:p4} are unchanged.

\subsection{From the uniformly bounded rate case to the unbounded case}\label{sec:ProofExists}

We now prove the reversibility results without the bounded rate assumption. The argument below applies to both discrete and Lebesgue (and even general) case. The idea is that any PKS satisfying the reversibility assumptions can be constructed as a limit of reversible PKSs with uniformly bounded rates. This proves at the same time that this PKS is well-defined and reversible. 

\medskip

Let $L$ be a PKS with parameters $(\lambda_0,\lambda_V,\lambda_H,\allowbreak p_0, p_V, p_H,\allowbreak \tau_V,\tau_H,F)$ which satisfies the assumptions of the main theorems, under the initial condition $(\mathcal{C}_X,\mathcal{C}_Y)$ as defined in equation~\eqref{eq:CXCY-g}. For any $n \geq 1$, we define the set
\begin{equation*}
\mathcal{S}(n) := \left\{ s\in \RR,\, \sup(\lambda_V(s), \lambda_H(s), \tau_V(s), \tau_H(s)) > n \right\}.
\end{equation*}
In words, the set $\mathcal{S}(n)$ is the set of weights for which the split and turn rates are larger than $n$. Remark that
\begin{displaymath}
\lim_{n\to\infty} \downarrow  \mathcal{S}(n) = \bigcap_{n \geq 1} \mathcal{S}(n) = \emptyset.
\end{displaymath}
This is due to the fact that the rate functions $\lambda_V$, $\lambda_H$, $\tau_V$ and $\tau_H$ are never equal to $+\infty$ by the hypotheses given by equations~\eqref{eq:notpos-g}, \eqref{eq:tauRev-g} and~\eqref{eq:lambdaRev-g}. Now, we define the following notation: for any $f: \RR \to \RR_+$,
\begin{displaymath}
f^{(n)} (s) := f(s) \ind{s \notin \mathcal{S}_n},
\end{displaymath}

We denote by $L^{(n)}$ the PKS with parameters $(\lambda_0,\lambda^{(n)}_V,\lambda^{(n)}_H,\allowbreak p_0, p^{(n)}_V, p^{(n)}_H,\allowbreak \tau^{(n)}_V,\tau^{(n)}_H,F)$ with initial condition $(\mathcal{C}_X,\mathcal{C}_Y)$. The PKS $L^{(n)}$ has rates uniformly bounded by $n$ and satisfies the hypotheses of the theorem \emph{for the same measures $\nu_V$ and $\nu_H$}.

Hence, by the version of Theorem~\ref{thm:reversible} we have proved, which assumes the boundedness of the rate functions, the PKS $L^{(n)}$ is well defined a.s.\ and reversible for the line weight measures $\nu_V$ and $\nu_H$. 

Now, we consider the process $L^{(n)}$ in the box $[0,a]\times [0,b]$. We want to estimate the number of lines of $L^{(n)}$ with weight $s \in \mathcal{S}(n)$ in this box, i.e.\ the number of lines that, without truncation, would have split or turn rates greater than $n$. For that, we will count the mean number of nodes in the box which have at least one edge with a weight in $\mathcal{S}(n)$.
 
As the pair of independent PPPs with intensities $\leb \otimes \tilde{\nu}_V$ and $\leb \otimes \tilde{\nu}_H$ is a stationary probability measure of $L^{(n)}$, for any small element of size $\di x \times \di y$, the probability to see a node such that the weight of its south edge or its west edge is an element of $\mathcal{S}(n)$ is 
\begin{displaymath}
\Big(\nu_V(\mathcal{S}(n)) \nu_H(\RR) + \nu_V(\mathcal{S}(n)) \nu_H(\RR) - \nu_V(\mathcal{S}(n)) \nu_H(\mathcal{S}(n))\Big) \, \di x \, \di y.
\end{displaymath}

By integration on the box $[0,a] \times [0,b]$, we find that the mean number of lines with weight in $\mathcal{S}(n)$ and which are a south or west edge of a node is
\begin{displaymath}
\Big( \nu_V (\mathcal{S}(n)) \, \nu_H(\RR) + \nu_H(\mathcal{S}(n)) \, \nu_V(\RR) \Big) a b.
\end{displaymath}

Moreover, because $L^{(n)}$ is reversible, this quantity is also equal to the mean number of lines with weight in $\mathcal{S}(n)$ and which are a north or east edge of a node. By summing these two means and adding the mean number of lines that are entering in the box, we deduce the following upper bound
\begin{align*}
& \esp{\text{number of segments of $L^{(n)}$ inside $[0,a]\times [0,b]$ whose weight $s \in \mathcal{S}(n)$}} \\
& \quad \leq \varepsilon(n) := 2 \Big(\nu_V(\mathcal{S}(n))\nu_H(\RR) + \nu_H(\mathcal{S}(n))\nu_V(\RR)\Big) a b + \Big(a \nu_V(\mathcal{S}(n)) + b \nu_H(\mathcal{S}(n))\Big).
\end{align*}

In particular, because it is an integer-valued random variable:
\begin{displaymath}
\prob{\text{there exists a segment in $L^{(n)}$ of weight $s \in \mathcal{S}(n)$}} \leq \varepsilon(n).
\end{displaymath}

As $\mathcal{S}(n) \to \emptyset$ when $n \to \infty$, we have that $\varepsilon(n) \to 0$. But, if $L^{(n)}$ does not contain any line whose weight is in $\mathcal{S}(n)$, then $L^{(n)}$ and $L$ coincide for the trivial coupling. Then we deduce that
\begin{align*}
& \prob{\text{the PKS $L$ is well defined, i.e.\ it does not explode, inside the box $[0,a]\times [0,b]$}}\\
& \geq \prob{\text{the PKS $L^{(n)}$ does not have any segment of weight $s \in \mathcal{S}(n)$}} \geq 1 - \varepsilon(n).
\end{align*}
This is true for all $n$, so the PKS $L$ is well defined a.s..  Moreover, the process $L$ is reversible with line weight measures $\nu_V$ and $\nu_H$ since it coincides with probability converging to $1$ with the reversible process $L^{(n)}$.

\section{Examples}\label{sec:examples}
\subsection{Potential function of a PKS} 
\begin{figure}
    \begin{center}
    \includegraphics[height=8cm]{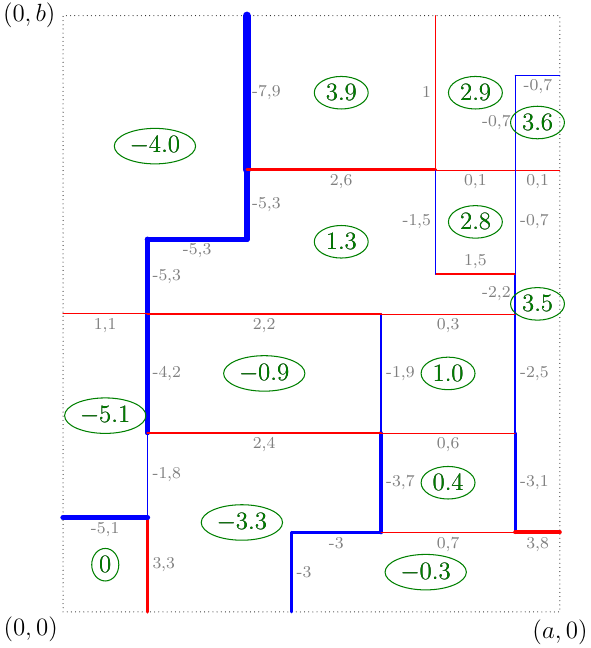}
    \end{center}
    \caption{An example of the potential on each connected component of a drawing.}
    \label{fig:pot_draw}
\end{figure}

\begin{figure}
    \begin{center}
    \subfloat[2D visualisation]{\includegraphics[height=7cm]{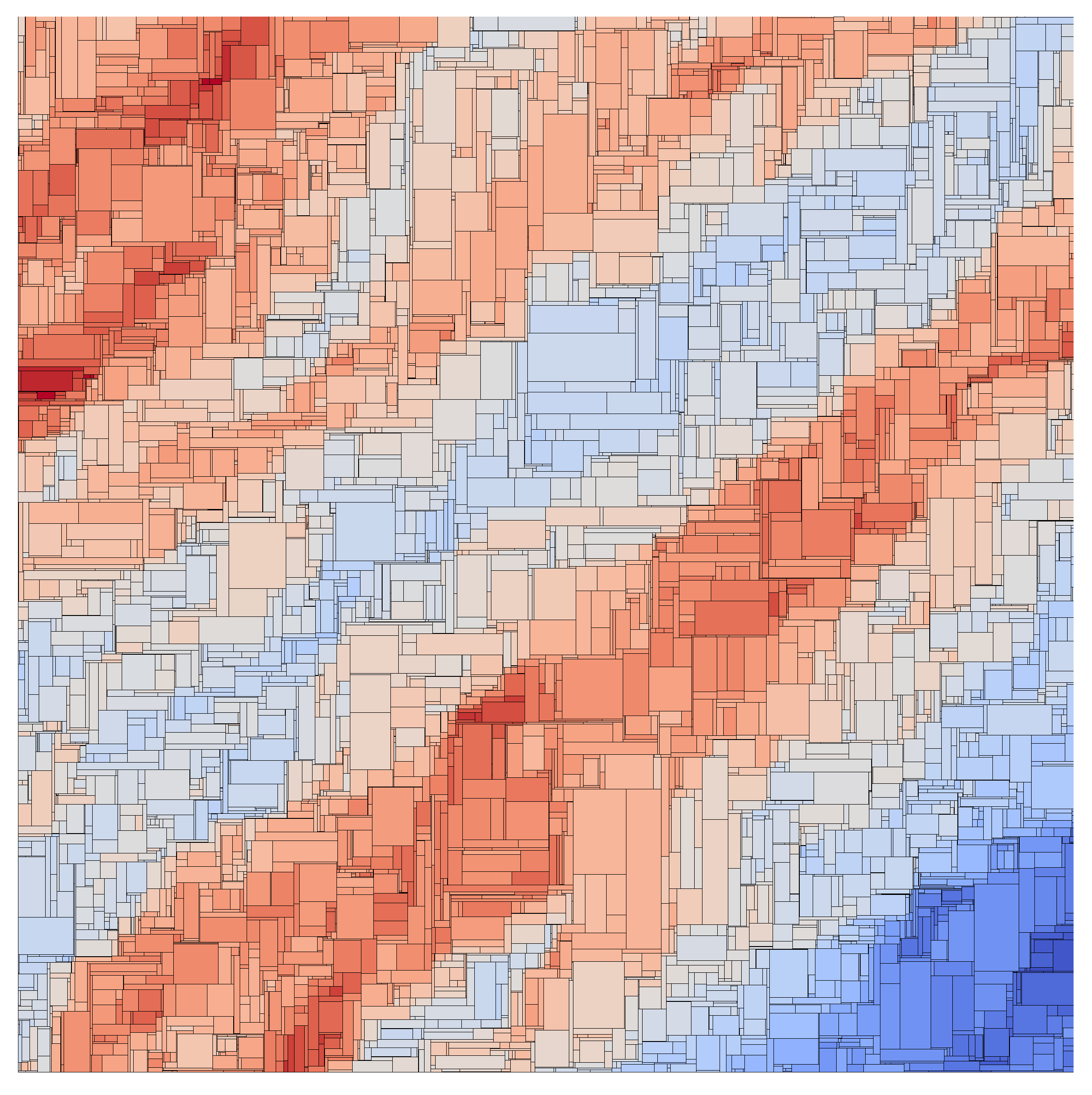}}
    \qquad
    \subfloat[3D visualisation]{\includegraphics[height=7cm]{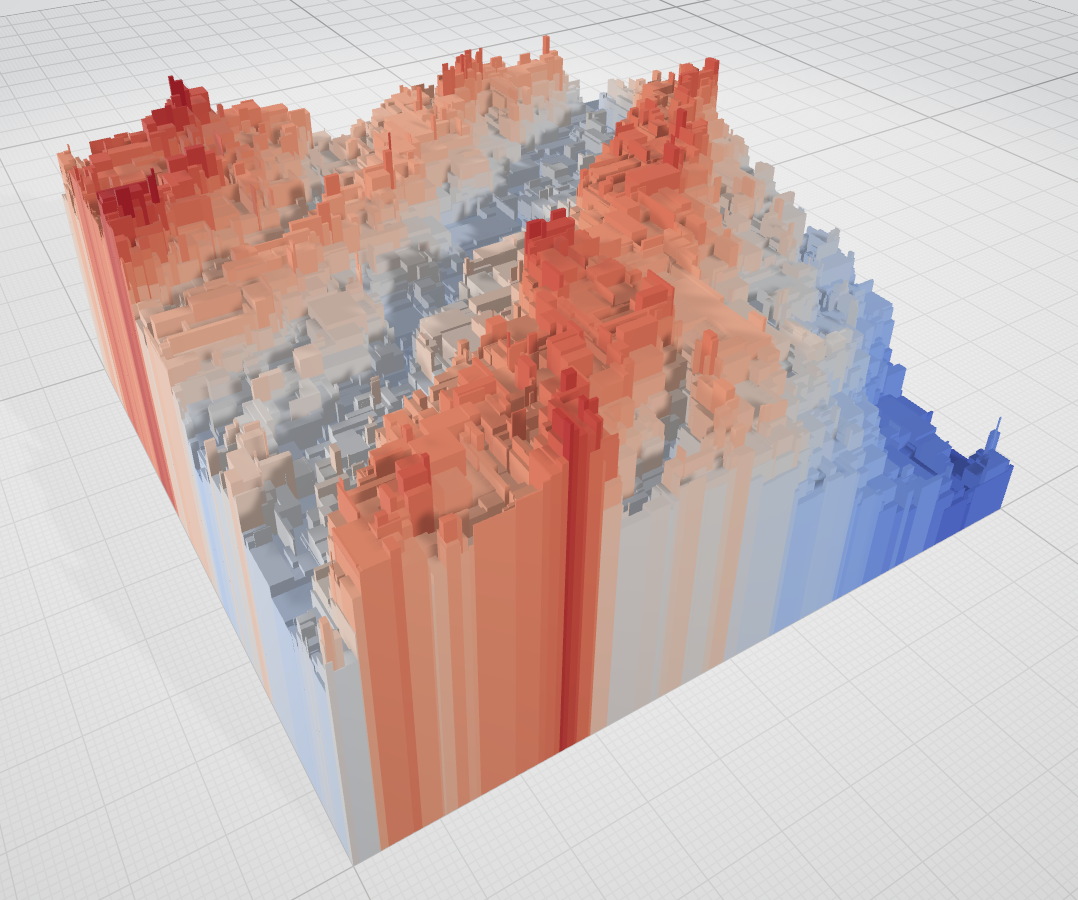}}
    \end{center}
    \caption{Simulation of a reversible PKS on $[0,80]\times [0,80]$ with parameters $p_V(s) = 0.5$, $p_H(s) = 0.5$, $\tau_V(s) = \tau_H(s) = 0$ and with line weight measures $\nu_V = \nu_H = \norm{0,1}$ (Model~\ref{ex:NormNorm} in Table~\ref{tab:examples}). Colors represent potential values: blue for negative ones and red for positive ones.}
    \label{fig:pot_norm}
\end{figure}

By construction, a PKS induces a random tessellation of the quarter plane into polygonal regions (which are the connected components obtained after removing the lines of the process). We call these connected components the \emph{faces} of the tessellation. The fact that a PKS satisfies Kirchhoff's node law at every intersection is equivalent to the existence of a potential function associated with the faces of the random tessellation. More precisely, we can associate  to each face $F$ a scalar value $v(F)$ in such way that the following holds true: 
\begin{itemize}
\item Let $\sigma$ denote a horizontal segment in the PKS with weight $s(\sigma)$. This segment separates two faces of the tessellation. Let $F$ denote the face \emph{below} $\sigma$ and let $F'$ denote the face \emph{above} $\sigma$. Then, it holds that
\begin{equation} \label{eq:pot-plus}
v(F') - v(F) = s(\sigma).
\end{equation}
\item Let $\sigma$ denote a vertical segment in the PKS with weight $s(\sigma)$. This segment separates two faces of the tessellation. Let us $F$ denote the face \emph{on the left} of $\sigma$  and let $F'$ denote the face \emph{on the right} of $\sigma$. Then, it holds that
\begin{equation}\label{eq:pot-minus}
v(F') - v(F) = -s(\sigma).
\end{equation}
\end{itemize}

In other words, equation~\eqref{eq:pot-plus} states that crossing a horizontal segment by moving upward \emph{increases} the potential by the value of the weight of the segment. On the other hand, equation~\eqref{eq:pot-minus} states that crossing a horizontal segment while moving to the right \emph{decreases} the potential by the value of the weight of this segment. See Figure~\ref{fig:pot_draw} for an illustration. 

The consistency of equations~\eqref{eq:pot-plus} and~\eqref{eq:pot-minus} for any segment is straightforward thanks to Kirchhoff's node law: looking at Figure \ref{fig:Kirchhoff}, we simply check that the sum of the potential differences when going (say clockwise) around a node is $s_W - s_N - s_E + s_S = 0$. Furthermore, it is clear that the potential function $v$ is unique up to an additive constant. By convention, we choose it to be $0$ for the bottom left face containing the origin. Figure~\ref{fig:pot_norm} shows 2D and 3D representation of the potential function for a PKS process obtained by simulation with Gaussian line weights. 

\subsection{List of examples} 
As explained above, a PKS can be seen either as a weighted line process or as a potential function on faces of a random tessellation. These dual points of view make it possible to recover several well-known models appearing in the statistical physics literature, in particular classical models related to Last Passage Percolation (LPP) as explained in Section~\ref{sec:LPP}.

In the rest of the section, we compute the parameter of the reversible PKS with line weight measures $\nu_V$ and $\nu_H$ for many usual continuous and discrete distribution. A list of examples of reversible PKS is presented in Table~\ref{tab:examples}. Subsequently, we discuss further some of the models in this list that enjoy special properties and that are connected to well-known models.

The usual distributions and their parametrization considered here are summarized in Table~\ref{tab:distr}. 

\begin{table}
    \begin{center}
    \begin{tabular}{|>{\centering\arraybackslash}m{.16\textwidth}|c|c|c|c|}
    \hline
    Name & Parameters & Notation & Support & Density \\
    \hline\hline
    Dirac & $a\in \RR$ & $\dirac{a}$ & $\{a\} $ & $\ind{x=a}$ \\
    \hline
    Discrete Uniform & $A \subset \ZZ$ & $\unif{A}$ & $A$ & $\dfrac{1}{|A|}\ind{x\in A}$ \\
    \hline
    Bernoulli & $p\in [0,1]$ & $\ber{p}$ & $\{0,1\} $ & $(1-p) \ind{x=0} + p\ind{x=1}$ \\
    \hline
    Binomial & $(n,p)\in \NN \times [0,1]$ & $\bin{n,p}$ & $\llbracket 0,n \rrbracket $ & $\dbinom{n}{x}p^x(1-p)^{n-x}$ \\
    \hline
    Geometric & $p\in [0,1]$ & $\geom{p}$ & $\NN$ & $p(1-p)^x$ \\
    \hline
    Poisson & $\lambda \in \RR_+$ & $\poi{\lambda}$ & $\NN$ & $e^{-\lambda} \dfrac{\lambda^x}{x!}$ \\
    \hline \hline
    Uniform & $[a,b] \subset \RR$ & $\unif{[a,b]}$ & $[a,b]$ & $\dfrac{1}{b-a}\ind{a \leq x \leq b}$ \\
    \hline
    Exponential & $\lambda \in \RR_+^*$ & $\expo{\lambda}$ & $\RR_+$ & $\lambda e^{-\lambda x} \ind{x\geq 0}$ \\
    \hline
    Gamma & $(k,\theta) \in \RR \times \RR_+^*$ & $\gam{k,\theta}$ & $\RR_+$ & $\dfrac{1}{\Gamma(k)\theta^k}x^{k-1}e^{-x/\theta} \ind{x\geq 0}$ \\
    \hline
    Normal (or Gaussian) & $(\mu,\sigma^2) \in \RR\times \RR_+$ & $\norm{\alpha,\beta}$ & $\RR$ & $\dfrac{1}{\sqrt{2\pi \sigma^2}} \exp\left( - \dfrac{(x-\mu)^2}{2\sigma^2} \right)$ \\
    \hline
    Beta & $(\alpha,\beta) \in (\RR_+^*)^2$ & $\bet{\alpha,\beta}$ & $[0,1]$ & $\dfrac{\Gamma(\alpha+\beta)}{\Gamma(\alpha)\Gamma(\beta)}x^{\alpha-1}(1-x)^{\beta-1}$ \\
    \hline \hline \hline
    The opposite of the measure & $\mu$, a real measure & $-\mu$  & $- \mathrm{Support}(\mu)$  & $\di \mu(-x)$ \\
    \hline
    \end{tabular}
    \caption{Parametrization of the usual distributions. Here, $\NN$ denotes the set $\{0,1,\dots\}$.}
    \label{tab:distr}
    \end{center}    
\end{table}

\newgeometry{left=1mm,right=1mm,bottom=1mm,top=1mm}
\begin{landscape}
 \pagestyle{empty}
\begin{table}[!p]
    \begin{center}
    \begin{tabular}{|c|c|c|c|c|c|}
        \hline
        Model & $\nu_V/\nu_V(\RR)$ & $\nu_H/\nu_H(\RR)$ & $\lambda_V(s)/(\nu_H(\RR)p_V(s))$ & $\lambda_H(s)/(\nu_V(\RR)p_H(s))$ & $F(s,\cdot)$ \\
        \hline
        \hline
        \numex \label{ex:Dirac0Dirac0} & $\dirac{0}$ & $\dirac{0}$ & $1$ & $1$ & $\dirac{0}$\\
        \hline
        \numex \label{ex:DiracaDiracb} & $\dirac{a}$ with $a \neq 0$ & $\dirac{b}$ with $b \neq 0$ & $0$ & $0$ & $\dirac{b}$
        \\
        \hline
        \numex & $\dirac{a}$ with $a \neq 0$ & $\dirac{0}$ & $1$ & $0$ & $\dirac{0}$
        \\
        \hline \hline
        \numex \label{ex:BerBer} & $\ber{q_V}$ & $\ber{q_H}$ & $(1-q_H)\ind{s=0} $ & $(1-q_V)\ind{s=0} $ & $\ber{\frac{q_H(1-q_V)}{q_V+q_H-2q_Vq_H}}\ind{s=1}$\\
        &  &  & ${}+ (1+q_H(q_V^{-1}-2))\ind{s=1}$ &${}+ (1+q_V(q_H^{-1}-2))\ind{s=1}$  & ${}+\dirac{0}\ind{s=0} + \dirac{1} \ind{s=2}$\\
        \hline 
        \numex \label{ex:-BerBer} & $-\ber{q_V}$ & $\ber{q_H}$ & $(1-q_H)\ind{s=-1} $ & $(1-q_V)\ind{s=1} $ &$\ber{\frac{q_Vq_H}{1-q_V-q_H+2q_Vq_H}}\ind{s=0}$ \\
         &  &  & ${}+ (1+q_H\frac{2q_V-1}{1-q_V})\ind{s=0}$ &${}+ (1+q_V\frac{2q_H-1}{1-q_H})\ind{s=0}$  &${}+\dirac{0}\ind{s=-1} + \dirac{1} \ind{s=1}$ \\
        \hline
        \numex & $\unif{[0,a]}$ & $\unif{[0,b]}$ & $\displaystyle \min\left(\frac{s}{b},1 \right)$ & $\displaystyle \min\left(\frac{s}{a},1 \right)$ & $\unif{[\max(0,s-a),\min(b,s)]}$\\
        \hline
        \numex \label{ex:-UnifUnif} & $-\unif{[0,a]}$ & $\unif{[0,b]}$  & $\displaystyle \min\left(\frac{a+s}{b},1 \right)$ & $\displaystyle \min\left(\frac{b-s}{a},1 \right)$ & $\unif{[\max(0,s),\min(b,s+a)]}$\\
        \hline \hline
        \numex \label{ex:GeomGeom} & $\geom{q}$ & $\geom{q}$ & $(s+1)\,q$& $(s+1)\,q$ & $\unif{\{ 0,\ldots, s \}}$ \\
        \hline
        \numex & $\expo{\gamma}$ & $\expo{\gamma}$ & $\gamma s$  & $\gamma s$ & $\unif{[0,s]}$\\
        \hline
        \numex \label{ex:ExpGeom} & $\expo{\gamma}$ & $\geom{q}$ & $\displaystyle q \sum_{t=0}^{\lfloor s \rfloor} \left((1-q)e^{\gamma}\right)^t$ &  $ 0$ & $\displaystyle \propto \sum_{t = 0}^{\lfloor s \rfloor} ((1-q) e^{\gamma})^t \dirac{t}$\\
        \hline
        \numex\label{ex:GammaGamma} & $\gam{k_V,\theta}$ & $\gam{k_H,\theta}$ & $\displaystyle \frac{\Gamma(k_V)}{\Gamma(k_V + k_H)} \left(\frac{s}{\theta}\right)^{k_H}$ & $\displaystyle \frac{\Gamma(k_H)}{\Gamma(k_V + k_H)} \left(\frac{s}{\theta}\right)^{k_V}$ & $ s\,\bet{k_V,k_H}$\\
        \hline \hline 
        \numex \label{ex:-BerGeom} & $-\ber{q_V}$ & $\geom{q_H}$ & $\dfrac{q_H(1-q_Vq_H)}{(1-q_V)}\ind{s=0} + q_H\ind{s=-1}$ & $(1-q_Vq_H)$ & $s + \ber{\dfrac{q_V(1-q_H)}{1-q_Vq_H}} \ind{s\geq 0} + \dirac{1}\ind{s=-1}$ \\
        \hline
        \numex \label{ex:-GeomGeom} & $- \geom{q_V}$ & $\geom{q_H}$ & $\displaystyle \frac{q_H}{q_V+q_H-q_Vq_H}$ & $\displaystyle \frac{q_V}{q_V+q_H-q_Vq_H}$ & $\max(0,s) + \geom{q_V+q_H-q_V q_H}$ \\
        \hline
        \numex \label{ex:-ExpExp}& $- \expo{\gamma_V}$ & $\expo{\gamma_H}$ & $\dfrac{\gamma_H}{\gamma_V+\gamma_H}$  & $\dfrac{\gamma_V}{\gamma_V+\gamma_H}$ & $\max(0,s) + \expo{\gamma_V+\gamma_H}$ \\
        \hline
        \numex \label{ex:-ExpGeom}& $- \expo{\gamma}$ & $\geom{q}$ & $\dfrac{q}{1-(1-q)e^{-\gamma}}$ & $0$ & $\max(0,\lceil s \rceil) + \geom{1-(1-q)e^{-\gamma}}$\\
        \hline \hline
        \numex \label{ex:NormNorm}& $\norm{0,1}$ & $\norm{0,1}$ & $\displaystyle \frac{1}{\sqrt{2}} \exp\left( \frac{s^2}{4}\right)$ & $\displaystyle \frac{1}{\sqrt{2}} \exp\left( \frac{s^2}{4}\right)$ &  $\displaystyle \norm{\frac{s}{2},\frac{1}{2}}$ \\
        \hline \hline
        \numex \label{ex:PoiPoi} & $\poi{\gamma_V}$ & $\poi{\gamma_H}$ & $\displaystyle e^{-\gamma_H} \left(1+\frac{\gamma_H}{\gamma_V}\right)^{\!s}$ & $\displaystyle e^{-\gamma_V} \left(1+\frac{\gamma_V}{\gamma_H}\right)^{\!s}$  & $\displaystyle \bin{s,\frac{\gamma_H}{\gamma_V+\gamma_H}}$\\
        \hline
    \end{tabular}
    \caption{Some remarkable PKS.}
    \label{tab:examples}
    \end{center}
\end{table}
\end{landscape}

\restoregeometry

\subsection{Models with monotone potential (LPP)}\label{sec:LPP}
For a PKS which satisfies the hypotheses of the main theorems, the weights of the vertical lines (resp.\ horizontal lines) take their values in the support of $\nu_V$ (resp.\ $\nu_H$) a.s.. It follows that the potential function $v$ is monotone in its both coordinates when the measures $\nu_V$ and $\nu_H$ have their support included in $\RR_+$ and $\RR_-$ respectively, or the opposite. For instance, if $\hbox{Support}(\nu_V) \subset \RR_-$  and $\hbox{Support}(\nu_H) \subset \RR_+$, then the potential $v$ is non-decreasing in both its coordinates. This is the case for Model~\ref{ex:-GeomGeom} in Table~\ref{tab:examples} simulated in Figure~\ref{fig:sim-discrete_b}. Another example is that of Model~\ref{ex:-ExpExp} whose associated potential is represented in the simulation on the front page of this paper. 

As it turns out, such PKS with monotone potentials can often be mapped to LPP models.

\paragraph{Standard Hammersley's model.} The standard Hammersley's broken line process described in the introduction of the paper and illustrated in Figure~\ref{fig:CG-fleche} (and studied for instance in~\cite{Hammersley72,AD95,Groeneboom02,CG05,CG06}) is clearly a PKS whose dynamics is the one of Model~\ref{ex:DiracaDiracb} with $a=-1$, $b=1$ and $p_0=1$.

\paragraph{Hammersley interacting fluid system.} In~\cite{CP12,CPS12}, a generalisation of the standard Hammersley's model is introduced. As for the standard Hammersley's model, it starts with a unit intensity PPP on $[0,\infty)^2$ but here, to each atom of the PPP is also associated a random positive number, chosen in i.i.d fashion, with common probability distribution $\mathcal{F}$ on the positive real numbers. At each atom of the PPP, a particle of positive weight equal to this number goes to the right and a particle of opposite weight goes up. When two particles of opposite weights collide, they both disappear; otherwise the particle of maximal absolute weight continues with a weight equal to the sum of the weights before the encounter and the other one disappears. When $\mathcal{F}$ is the Dirac measure at $1$, we recover the classical Hammersley process. 

The Hammersley fluid model is a PKS with parameters $\lambda_0 = 1$, $\lambda_V = \lambda_H =0$, $\tau_V = \tau_H =0$, $p_0 = 1$, $p_V(s) = \ind{s < 0}$, $p_H(s) = \ind{s>0}$ and $F(0;.) = \mathcal{F}(.)$. Note that we do not need to specify $F(s,.)$ for $s \neq 0$ because $p_V(s) + p_H(s) = 1$ and $\lambda_V(s)=\lambda_H(s)=0$. However, in view of Theorems~\ref{thm:reversible} and~\ref{thm:reversible-d}, the only model in this class that is reversible corresponds to the usual Hammersley's process when the law $\mathcal{F}$ is a Dirac law.

\medskip

We can also recover, within the framework of PKS, other LPP systems defined on the discrete grid $\NN^2$ by embedding these models on $\RR_+^2$. In order to do so, we impose the functions $p_V$, $p_H$ as well as the constant $p_0$ to be all identically $0$. This assumption ensures that splits, annihilations and creation events can never happen during the dynamics. Thus, each line on the initial PPPs on the $x$- and $y$-axis survives forever and the trace of all those lines define a $2$-dimensional discrete grid (embedded in $\RR_+^2$ with exponential spacing). 
 
\paragraph{Exponential Last Passage Percolation.}
Model~\ref{ex:-ExpExp} where $-\nu_V$ and $\nu_H$ are proportional to exponential distributions corresponds to the Exponential LPP studied by~\cite{Rost81}.

\paragraph{Geometric Last Passage Percolation.}
There are two cases of Geometric LPP depending on whether the geometric distributions start from $0$ or $1$. Both cases have been studied~\cite{CEP96,Martin06,Seppalainen09} and both can be seen as special PKS models. In the Table~\ref{tab:examples}, we have only detailed the case of the geometric starting from $0$, see Model~\ref{ex:-GeomGeom}.

\paragraph{Discrete Hammersley's processes.} In~\cite{BEGG16}, two models are defined. Their second model corresponds to Model~\ref{ex:-BerGeom}, mixing a discrete Bernoulli distribution with a Geometric distribution. In that case, with the notation of \cite{BEGG16}, the probability that a site contains a ``cross'' is equal to $p=\dfrac{q_V(1-q_H)}{1-q_V q_H}$. Let us note that their first model however \emph{cannot} be mapped to a PKS. Indeed, the model is still conservative (i.e.\ it obeys Kirchhoff law), but the transition  kernel at a crossing of lines depends not only on the total incoming weights but also on the horizontal/vertical division of the global weight. Thus, in order to encompass this first model, one would need to significantly generalize the definition of a PKS process. This is doable but it lies outside the scope of this paper. 

\paragraph{Generalised Last Passage Percolation.}
Let $\mu_0$ be a probability measure on $\NN^*$ (resp.\ on $\RR_+$ with density $f_0$). Taking, for any $A \in \borel{\RR_+}$, $\nu_V(-A) = \nu_H(A) = \sum_{i \in A} \sqrt{\mu_0(i)}/Z$ where $Z = \sum_{i \in \NN^*} \sqrt{\mu_0(i)}$ (resp. $\nu_V(-A) = \nu_H(A) = \int_{A} \sqrt{f_0(s)} \di s/Z$ where $Z= \int_{\RR^*} \sqrt{f_0(s)} \di s$), we recover the generalised LPP defined in~\cite{Casse19}. 

\subsection{Bullet models} We call \emph{bullet models} the family of models where the weight of the lines plays no role. These models can be obtained by taking $\nu_V$ and $\nu_H$ proportional to $\dirac{0}$ (see Model~\ref{ex:Dirac0Dirac0} of the Table~\ref{tab:examples}), and the turn rate functions $\tau_V = \tau_H =0$. Some of them have been lingering in the scientific community for a few years and are notoriously difficult to study out of equilibrium, see for instance~\cite{BM20,HST21} and references therein. In~\cite{BCGKP15}, some bullet models are proved to be stationary. Some of them correspond to PKS processes with specific parameters. For instance, setting $\nu_V$ and $\nu_H$ to be proportional to $\dirac{0}$ and assuming that $p_0 = 0$, the dynamics of the bullets can be formulated as follows. When an horizontal and a vertical bullets meet:
\begin{itemize}
        \item with probability $p_V(0)$, the horizontal bullet is destroyed and the vertically one continues its course. 
        \item with probability $p_H(0)$, the vertical bullet is destroyed and the horizontal one continues its course. 
        \item with probability $1-p_V(0)-p_H(0)$, both bullets continue their course (passing through each other). 
\end{itemize}

Of course, taking the Dirac at $0$ for the weight measure may appear as cheating a somewhat since the potential associated with the PKS is then constant to $0$. However, it is also possible to define bullet models by choosing $-\nu_V$ and $\nu_H$ to be Geometric or Exponential measures instead of a Dirac at $0$. Indeed, for all these measures, it follows from Models~\ref{ex:-GeomGeom}, \ref{ex:-ExpExp} and~\ref{ex:-ExpGeom}, that the splitting rate of lines remains constant (i.e.\ does not depend on the weight of the line). Thus, interestingly enough, this shows that some bullet model can be interpreted as the trace of more complex, non-trivial, potential models. 

\subsection{Six-vertex model}
The six-vertex model is a standard model in statistical physics introduced first by Pauling in 1935~\cite{Pauling35} to study the ice in two dimensions. From a mathematical point of view, it is a family of probability measures on the set of orientations of the grid $N \times N$, such that there are only two incoming edges around each node. Hence, there are only six possible local configurations allowed. To each type $i$ of a local configuration, we associate a weight (an ``energy'') $w_i$, see Table~\ref{tab:six-config}. From these weights, we can define a probability measure on the set of orientations of the grid $N^2$ via the following formula: for any orientation $O$,
\begin{equation*}
    \prob{O} = \frac{1}{Z} \prod_{(x,y) \in [0,N]^2} w_{\text{type}_O(x,y)}
\end{equation*}
where $Z = \sum_{O}  \prod_{(x,y) \in [0,N]^2} w_{\text{type}_O(x,y)}$ and where $\text{type}_O(x,y) \in \{1,\dots,6\}$ denote the type of the local configuration seen around the point $(x,y)$ in the orientation $O$, see Figure~\ref{fig:six-vertex}.

\begin{figure}
    \begin{center}
    \includegraphics{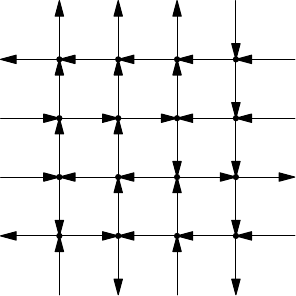}
    \caption{A configuration of the six-vertex model in the grid $4 \times 4$. Its probability associated is $\frac{1}{Z}\, w_1^2\, w_2 ^3\, w_3^1\, w_4^5\, w_5^3\, w_6^2 $.}
    \label{fig:six-vertex}
    \end{center}
\end{figure}

Usually, the model is studied with the assumption that \emph{there does not exist an external electromagnetic field} that implies that $w_1=w_2=a$, $w_3=w_4=b$ and $w_5=w_6=c$. Such models of six and also eight-vertex models have been deeply studied, and we refer the interested reader to \cite{FW70,Sutherland70,Baxter72,Baxter82,KDN90, BCG16,Casse18,DGHMT18,Melotti21} and references therein.\par
Some six-vertex models \emph{with an external electromagnetic field} turn out to be special cases of PKS from Models~\ref{ex:BerBer} and~\ref{ex:-BerBer} where $p_V=p_H=\tau_V=\tau_H=0$ (to get only crossings). Namely, we can construct a six-vertex model from PKSs of type~\ref{ex:BerBer} and~\ref{ex:-BerBer}, in the following way:
\begin{itemize}
    \item In the case of Model~\ref{ex:BerBer}, to any horizontal segment with weight $0$ (resp.\ $1$) of the PKS, we associate an oriented segment to the west (resp.\ to the east) in the six-vertex configuration; and similarly, to any vertical segment with weight $0$ (resp.\ $1$) of the PKS, we associate an oriented segment to the south (resp.\ to the north) in the six-vertex model.
    \item In the case of Model~\ref{ex:-BerBer}, to any horizontal segment with weight $0$ (resp.\ $-1$) of the PKS, we associate an oriented segment to the east (resp.\ to the west) in the six-vertex configuration; and similarly, to any vertical segment with weight $0$ (resp.\ $1$) of the PKS, we associate an oriented segment to the south (resp.\ to the north) in the six-vertex model. 
\end{itemize}
See Figure~\ref{tab:six-config} to get an illustration of these correspondences.

\begin{table}
    \begin{center}
      \begin{tabular}{>{\centering\arraybackslash}m{.12\textwidth} | >{\centering\arraybackslash}m{.11\textwidth} >{\centering\arraybackslash}m{.11\textwidth} >{\centering\arraybackslash}m{.11\textwidth} >{\centering\arraybackslash}m{.11\textwidth}
      >{\centering\arraybackslash}m{.11\textwidth} 
      >{\centering\arraybackslash}m{.11\textwidth}}
      Type & 1 & 2 & 3 & 4 & 5 & 6 \\
      \hline
      Six-vertex model & 
      \includegraphics{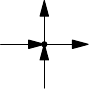} &  \includegraphics{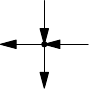} &  \includegraphics{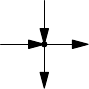} &  \includegraphics{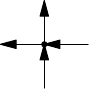} & 
      \includegraphics{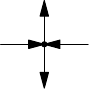} & 
      \includegraphics{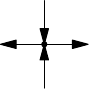} \\
      \hline
      ``Energy'' & $w_1$ & $w_2$ & $w_3$ & $w_4$ & $w_5$ & $w_6$ \\
      \hline \hline
      PKS Model~\ref{ex:BerBer} & \includegraphics{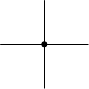} &  \includegraphics{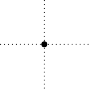} &  \includegraphics{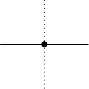} &  \includegraphics{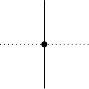} & 
      \includegraphics{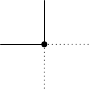} & 
      \includegraphics{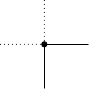} \\
      \hline
      ``Energy'' Model~\ref{ex:BerBer} & \small{$q_V +q_H -2q_Vq_H$} & \small{$q_V +q_H -2q_Vq_H$}  & \small{$q_H(1-q_V)$} & \small{$q_V(1-q_H)$} & 
      \small{$q_V(1-q_H)$} & \small{$q_H(1-q_V)$} \\
      \hline \hline
      PKS Model~\ref{ex:-BerBer} & \includegraphics{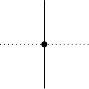} &  \includegraphics{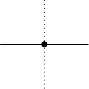} &  \includegraphics{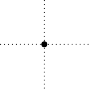} &  \includegraphics{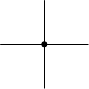} & 
      \includegraphics{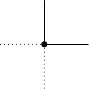} & 
      \includegraphics{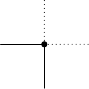} \\
      \hline
      ``Energy''  & \small{$1 -q_H -q_V$} & \small{$1 -q_H -q_V $} & \small{$q_Vq_H$} & \small{$(1-q_V)\cdot{}$} & 
      \small{$(1-q_V)\cdot{}$} & \small{$q_V q_H$} \\
      Model~\ref{ex:-BerBer} & \small{${}+ 2q_V q_H$} & \small{${}+ 2q_V q_H$} &  & \small{$(1-q_H)$} & 
      \small{$(1-q_H)$} & 
    \end{tabular}
    \caption{On the first line, the six local configurations allowed. On the second line, their local ``energy''. On the third line, their correspondence with a configuration of Model~\ref{ex:BerBer}: a plain line stands for a line with weight $1$ and a dotted line for a line with weight $0$. On the fourth line, their correspondence with a configuration of Model~\ref{ex:-BerBer}: a plain line stands for a line with weight $1$ or $-1$ and a dotted line for a line with weight $0$.}
    \label{tab:six-config}
    \end{center}
\end{table}

\subsection{Gaussian and Poisson models} \label{sec:GPM}
Models~\ref{ex:NormNorm} and~\ref{ex:PoiPoi} correspond to models with Gaussian and Poisson marginals respectively. Both models have a particularly nice explicit dynamics. These models are new to the best of our knowledge and look interesting to study further. In model~\ref{ex:PoiPoi}, one can show that the potential has the same distribution as the difference of two independent Poisson variables, which gives rise to a law of large number with fluctuation of order $n^{1/4}$ with asymptotically Gaussian distribution.

\section{Statistical properties of the tessellation}\label{sec:tess}
In this section, we look at the basic geometric properties of the system of lines generated by a reversible PKS satisfying the assumptions of Theorem~\ref{thm:reversible} or~\ref{thm:reversible-d}. We focus here our attention on the case where $p_0=0$. Indeed, if $p_0 \neq 0$, the number of faces could be sub-quadratic according to the length $a$ of a square $[0,a] \times [0,a]$ as it is the case for the Hammersley broken line process (presented in the Introduction) where the number of faces is linear in $a$.

\paragraph{Number of faces and nodes.}
Let $D \in \dessin_{a,b}$ be a drawing, we can associate to this drawing a \emph{tessellation} as the set of segments of $D$ without notifying their weight.
We call a \emph{face} of a tessellation $T$ a connected component of $([0,a] \times [0,b]) \setminus T$.

\begin{proposition}\label{prop:mn}
Consider a reversible PKS such that its initial condition $(\mathcal{C}_X,\mathcal{C}_Y)$ is distributed according to two independent PPPs respectively on $(\RR_+ \times \{0\}) \times \RR$ with intensity $\di x \, \di \nu_V(s)$ and on $(\{0\} \times \RR_+) \times \RR$  with intensity $\di y \, \di \nu_H(s)$ where $\nu_V$ and $\nu_H$ are two non-zero finite measures on $\RR$ satisfying conditions~\eqref{eq:lambda0-g}, \eqref{eq:notpos-g}, \eqref{eq:tauRev-g}, \eqref{eq:lambdaRev-g} and~\eqref{eq:fRev-g}. Then, the law of the tessellation associated to this PKS is translation-invariant and satisfies:

\begin{enumerate}[label=(\roman*)]
    \item The mean number of faces that do not touch the northern or eastern sides of $[0,a] \times [0,b]$, is $a b \, \nu_V(\RR) \nu_H(\RR)$. 
    \item The mean number of nodes of each type is summarized in the table below:
    \begin{center}
        \begin{tabular}{|c|c|} \hline
        Type &  Mean number \\ \hline\hline
        $\VE$  & \multirow{2}{*}{$a \, \nu_V(\RR)$} \\ \cline{1-1}
        $\VS$  &  \\ \hline
        $\HB$  & \multirow{2}{*}{$ab \int_{\RR} p_V(s)\, (\nu_V \ast \nu_H)(\di s)$} \\ \cline{1-1}
        $\HA$ &  \\ \hline
        $\HT$ & $ab \int_{\RR} \tau_V(s) \, \di \nu_V(s) = ab \int_{\RR} \tau_H(s) \, \di \nu_H(s)  $ \\
        \hline\hline
        $\CC$ & $ab \int_{\RR} (1-p_V(s)-p_H(s)) (\nu_V \ast \nu_H)(\di s)$ \\
        \hline
        \end{tabular}      
    \end{center}
    The mean number of horizontal nodes can be found by swapping $H$ and $V$, and $a$ and $b$.
    \item When $a$ or $b$ goes to infinity, the number of each type of node is almost surely asymptotically equal to their     rescaled mean. 
\end{enumerate}
\end{proposition}

\begin{proof} We will first determine the mean number of each node, and we will deduce the mean number of faces, so that we will first prove $(ii)$ and then $(i)$ and $(iii)$.
\begin{enumerate}[label=(\roman*),leftmargin=*]
\item[$(ii)$]
\begin{itemize}[leftmargin=*]
    \item Nodes of type $\VE$ or $\VS$: by definition, the nodes $\VE$ are distributed according to a PPP of intensity $\nu_V(\RR)$ on the $x$-axis. Thus, the mean number of such nodes on the segment $(0,a) \times \{ 0 \}$ is $a \, \nu_V(\RR)$. By Theorem~\ref{thm:reversible}, the same holds for the mean number of nodes of type $\VS$ on the segment $(0,a) \times \{ b \}$.
    \item Nodes of type $\HB$ or $\HA$: we do the proof for $\HA$. For any $(x,y)$, the probability to see in the box $[x,x+\di x] \times [y+\di y]$, a vertical line, a horizontal line and a vertical coalescence is 
    \begin{align*}
        \di x \, \di y \int_{\RR} \int_{\RR} p_V(s+t) \, \di \nu_V(s) \, \di \nu_H(t) & = \di x \, \di y \int_{\RR} \int_{\RR} p_V(u) \, \di \nu_V(s) \, \di \nu_H(u-s) \\
        & = \di x \, \di y \, \int_{\RR} p_V(u) \, (\nu_V \ast \nu_H)(\di u)
    \end{align*}
    Hence, the mean number of nodes on type $\HA$ in $[0,a] \times [0,b]$ is
    \begin{align*}
    \int_{[0,a]} \di x \int_{[0,b]} \di y \int_{\RR} p_V(s) \, (\nu_V \ast \nu_H)(\di s)
    = a b \int_{\RR} p_V(s)\, (\nu_V \ast \nu_H)(\di s).
    \end{align*}
    By Theorem~\ref{thm:reversible}, the same holds for the mean number of nodes of type $\HB$.
    \item Nodes of type $\HT$: similarly, the probability to see in the box $[x,x+\di x] \times [y,y+\di y]$ a vertical line of size $[s,s + \di s]$ and a vertical turn (that happens at rate $\tau_V(s)$) is
    \begin{equation*}
        \di x\, \di y\, \tau_V(s) \, \di \nu_V(s) = \di x\, \di y\, \tau_H(s) \, \di \nu_H(s).
    \end{equation*}
    We conclude by integration on $\RR$ for $s$, $[0,a]$ for $x$ and $[0,b]$ for $y$.
    \item Nodes of type $\CC$: similar to the case of $\HA$ where we multiply by $1-p_V(s+t)-p_H(s+t)$ instead of $p_V(s+t)$ since we are in the case $3$(d) of the dynamics of Section~\ref{sec:PKS}.
\end{itemize}

\item[$(i)$] Now, we have two ways to prove the mean number of connected components that do not touch the northern or eastern sides of the rectangle, which is the same as the one which do not touch the southern or western sides of the rectangle by Theorem~\ref{thm:reversible}. Just remark that any connected component has only one north-east corner and one south-west corner. Hence, the mean number of connected components is both equal to the mean number of nodes of types $\VB \cup \HB \cup \CC$, and to the mean number of nodes of types $\VA \cup \HA \cup \CC$.

\item[$(iii)$] We treat the case where $a$ is fixed and  $b \to \infty$. By monotonicity, we can assume that $b$ is integer-valued. 
Let $\vad$ be a random drawing on $[0,a] \times \RR_+$. For any integer $n \geq 0$, we denote by $M_n$, the restriction of $\vad$ to the segment $[0,a] \times \{n\}$. The process $(M_n)$ is a Markov chain and, according to Corollary~\ref{cor:revers}~$\ref{en:stat}$, it admits an invariant measure. Moreover, this chain is irreducible since the empty set (no line going up)  can be reached with positive probability from any starting configuration. Thus, according to classical results on Markov chains having an accessible atom (see~\cite[Section~15.1]{MT93} and~\cite[Theorem~1]{AG11}), this chain is Harris recurrent and the law of large numbers applies. \qedhere
\end{enumerate}
\end{proof}

\paragraph{Mean number of nodes and corners around a face.} 
For any positive $a,b$, let $\vad$ be a random drawing of $\dessin_{a,b}$. Denote by $\mathcal{F}(\vad)$ the set of faces of the tessellation of $\vad$. For a given face $F$ we denote by $s_F$ and $c_F$ the number of nodes and of corners (that are the nodes on the boundary of $F$ whose angle is $\pi/2$ or $3\pi/2$) around the face $F$. This is illustrated on Figure~\ref{fig:face}. Here, we are interested in $\mathbf{s}_{a,b}$ (resp.\ $\mathbf{c}_{a,b}$) the mean number of nodes (resp.\ corners) of $\mathcal{F}(\vad)$, namely
\begin{equation}
    \mathbf{s}_{a,b} = \frac{\sum_{F \in \mathcal{F}(\vad)} s_F}{\left| \mathcal{F}(\vad)\right|} \quad\left(\text{resp.\ }  \mathbf{c}_{a,b} = \frac{\sum_{F \in \mathcal{F}(\vad)} c_F}{\left| \mathcal{F}(\vad)\right|}\right).
\end{equation}

\begin{figure}
    \begin{center}
    \includegraphics[height=3.5cm]{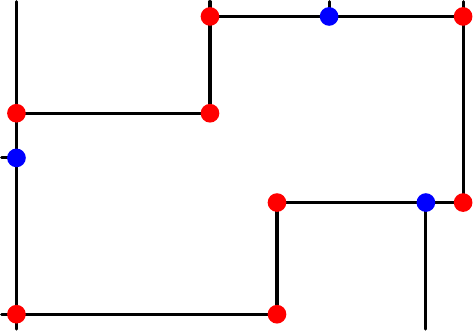}
    \caption{A connected component with $11$ nodes (in blue and red) and $8$ corners (in red).}\label{fig:face}
    \end{center}
\end{figure}

\begin{corollary}
The following almost sure limits hold:
\begin{enumerate}[label=(\roman*)]
 \item $\displaystyle
     \lim_{b \to \infty} \lim_{a \to \infty} \mathbf{s}_{a,b}  = \lim_{a \to \infty} \lim_{b \to \infty} \mathbf{s}_{a,b} = 4 + 2 \frac{\int_{\RR} \left(p_V(u)+p_H(u)\right) (\nu_V \ast \nu_H)(\di u) + 2\int_{\RR} \tau_V(u) \di \nu_V(u)}{\nu_V(\RR) \nu_H(\RR)}$.

\item $\displaystyle \lim_{b \to \infty} \lim_{a \to \infty} \mathbf{c}_{a,b} = \lim_{a \to \infty} \lim_{b \to \infty} \mathbf{c}_{a,b} =
     4 + 4  \frac{\int_{\RR} \tau_V(u) \di \nu_V(u)}{\nu_V(\RR) \nu_H(\RR)}$.
\end{enumerate}
\end{corollary}

\begin{remark}
When $p_V+p_H = 1$ and $\tau_V = \tau_H = 0$, all the nodes are of degree $3$ and all faces are rectangles. Hence, our result recover the well-known fact that the mean number of nodes per faces is equal to $6$ and the number of corners is obviously $4$.
\end{remark}

\begin{proof}
First, we use point $(iii)$ of Proposition~\ref{prop:mn} to go back and forth between a.s.\ convergence and convergence of mean. Notice that since $a$ and $b$ go to infinity, we do not care about counting the faces that touch the boundary of a finite rectangle $[0,a]\times[0,b]$ or the nodes on it, because their proportion, compared to the total number of node in the box, goes to zero as the box gets larger. 

\begin{enumerate}[label=$(\roman*)$,leftmargin=*]
    \item By definition:
    \begin{align*}
    \mathbf{s}_{a,b} & =  \frac{\sum_{F \in \mathcal{F}(\vad)} s_F}{\left| \mathcal{F}(\vad)\right|} = \frac{\sum_{F \in \mathcal{F}(\vad)} s_F/a}{\left| \mathcal{F}(\vad)\right|/a}.
    \end{align*}
    
    And so by $(iii)$ of Proposition~\ref{prop:mn},
    \begin{align*}
        \lim_{a \to \infty} \mathbf{s}_{a,b} & = \frac{{\displaystyle\lim_{a \to \infty}}\esp{\sum_{F \in \mathcal{F}(\vad)} s_F}/a}{\displaystyle \lim_{a \to \infty} \esp{\left| \mathcal{F}(\vad)\right|}/a}.
    \end{align*}
    
    But, by $(ii)$ of Proposition~\ref{prop:mn}, $\esp{\left|\mathcal{F}(\vad)\right|} = ab\, \nu_V(\RR) \nu_H(\RR)$ and 
    \begin{align*}
    \esp{\sum_{F \in \mathcal{F}(\vad)} \mathbf{s}_{F}} & = \esp{\text{number of nodes with multiplicity}}  \\
    & =  \esp{2(|\VE|+|\HE| + |\HT| + |\VT| ) + 3 (|\HB| +|\HA| + |\VB| + |\VA|) +4|\CC|}  \\
    & = 2a \, \nu_V(\RR) + 2b \, \nu_H(\RR) + 6 ab \int_{\RR} (p_V(u) + p_H(u)) \, (\nu_V \ast \nu_H)(\di u)  \\
    & \quad  + 4 ab \int_{\RR} (1-p_V(u)-p_H(u)) \, (\nu_V \ast \nu_H)(\di u) + 4 ab \int_{\RR}  \tau_V \, \di \nu_V(u) \\
    & = 4 ab \int_{\RR} (\nu_V \ast \nu_H)(\di u) + 2ab \int_{\RR} (p_V(u) + p_H(u)) (\nu_V \ast \nu_H)(\di u) \\
    & \quad + 2a\, \nu_V(\RR) + 2b\, \nu_H(\RR) + 4 ab \int_{\RR} \tau_V(u) \di \nu_V(u).
    \end{align*}
Consequently,    
    \begin{align*}
    \lim_{a \to \infty} \mathbf{s}_{a,b} & = \lim_{a\to\infty} \Bigg[ 2 \left( \frac{\int_{\RR} \left(p_V(u)+p_H(u)\right) (\nu_V \ast \nu_H)(\di u) + 2\int_{\RR} \tau_V(u) \di \nu_V(u)}{\nu_V(\RR) \nu_H(\RR)} \right) \\
    & \qquad \qquad \quad {} + 4 \underbrace{\frac{\int_{\RR} (\nu_V \ast \nu_H)(\di u)}{\nu_V(\RR) \nu_H(\RR)}}_{=1} + 2\left( b\,\nu_H(\RR) \right)^{-1} + 2\left( a\,\nu_V(\RR) \right)^{-1} \Bigg].
    \end{align*}

    Now, taking the limit when $b \to \infty$ leads to the wanted results for $\displaystyle \lim_{b \to \infty} \lim_{a \to \infty} \mathbf{s}_{a,b}$. The same holds for $\displaystyle \lim_{a \to \infty} \lim_{b \to \infty} \mathbf{s}_{a,b}$.
    \item Similarly for $\displaystyle \lim_{b \to \infty} \lim_{a \to \infty} \mathbf{c}_{a,b}$ and $\displaystyle \lim_{a \to \infty} \lim_{b \to \infty} \mathbf{c}_{a,b}$. The difference is the term in the numerator that becomes
    \begin{displaymath}
        \esp{2(|\VE|+|\HE| + |\HT| + |\VT| + |\HB| +|\HA| + |\VB| + |\VA|) +4|\CC|}. \qedhere
    \end{displaymath}
\end{enumerate}
\end{proof}

The presence of double limits in this corollary is difficult to avoid. One could wish to get results about averages taken on boxes $[0,a] \times [0,b]$ with $a$ and $b$ having the same order of magnitude. But this type of results is related to decorrelation properties of the process which are more and more difficult to prove as lines get closer to the diagonal, as we can see on Figure~\ref{fig:pot_norm}.
\bigskip

\section{Perspectives} \label{sec:end}
In this paper we defined the Poisson-Kirchhoff model as a system of vertical and horizontal broken weighted lines with a Markovian reversible dynamic that preserves Kirchhoff's node law. In doing so, we made several assumptions, some of which could be relaxed, yet might still lead to tractable (and still reversible) dynamics. For instance one could look at:

\begin{itemize}
    \item Models where the distribution of weights on crossing events depends on the value of the two entries and not only through their sums. By relaxing this condition, we would recover the first model of~\cite{BEGG16} as explained in Section~\ref{sec:LPP}.
    \item Models where the measures $\nu_V$ and $\nu_H$ may have infinite mass. Relaxing the finite mass assumption should make it be possible to construct systems which are self similar, i.e.\ invariant by re-scaling of both space and weights simultaneously. Such models will be obtained by choosing $g_V$ and $g_H$ of the form $s \mapsto s^{-\alpha}$.
    \item More generally, one could consider models whose lines are not vertical and horizontal anymore, but can instead make an angle with the axis. A trivial example, with deterministic dynamics, is the Crofton model also called Poisson Line Process~\cite{Crofton1868}.
\end{itemize}

Finally, another important question is that of the fluctuations of the potential for these models. For now, it is only known that the Hammersley processes belong to the KPZ family class and have fluctuations of order $n^{1/3}$~\cite{CG05,CG06}. On the other hand, it can be proved that Model~\ref{ex:PoiPoi} of Table~\ref{tab:examples} has Gaussian fluctuations as mentioned in Section~\ref{sec:GPM}. We believe that these are the only two regimes of fluctuations and that the type of fluctuations depends on the support of $\nu_V$ and $\nu_H$. We conjecture that if the supports of both $\nu_V$ and $\nu_H$ are respectively included, either in $\RR_+$ and $\RR_-$, or in $\RR_-$ and $\RR_+$, then the fluctuations should be Tracy-Widom but Gaussian otherwise.

\bigskip
\noindent \textbf{Acknowledgements:} We acknowledge support from \texttt{ANR 16-CE93-0003} ``MALIN''. The third and the fourth authors also acknowledge respectively the supports from \texttt{ANR 16-CE40-0016 } ``PPPP'' and from \texttt{ANR 19-CE40-0025} ``ProGraM''.

\bibliographystyle{alpha}
\bibliography{aaa}

\end{document}